\documentclass[10pt]{amsart}
\theoremstyle{plain}
\usepackage{amsmath, amsfonts, amsthm, amssymb,amscd,bbold}
\usepackage{graphics}
\usepackage{color}
\usepackage{epsfig}
\usepackage{hyperref}
\usepackage[all]{xy}
\usepackage{tikz}
\graphicspath{ {Figures/} }

\title[SKh and HH]{Sutured Khovanov homology, Hochschild homology, and the Ozsv{\'a}th-Szab{\'o} spectral sequence}
\author{Denis Auroux}
\thanks{DA was partially supported by NSF grant DMS-1007177.}
\address{UC Berkeley, Department of Mathematics, 970 Evans Hall \# 3840, Berkeley CA 94720, USA}
\email{auroux@math.berkeley.edu}

\author{J. Elisenda Grigsby}
\thanks{JEG was partially supported by NSF grant number DMS-0905848 and NSF CAREER award DMS-1151671.}
\address{Boston College; Department of Mathematics; 301 Carney Hall; Chestnut Hill, MA 02467, USA}
\email{grigsbyj@bc.edu}

\author{Stephan M. Wehrli}
\thanks{SMW was partially supported by NSF grant number DMS-1111680.}
\address{Syracuse University; Mathematics Department; 215 Carnegie; Syracuse, NY 13244, USA}
\email{smwehrli@syr.edu}

\theoremstyle{plain}
\newtheorem{theorem}{Theorem}[section]
\newtheorem{lemma}[theorem]{Lemma}
\newtheorem{proposition}[theorem]{Proposition}
\newtheorem{corollary}[theorem]{Corollary}
\newtheorem{conjecture}[theorem]{Conjecture}

\newtheorem*{theoremSKhisHH}{Theorem~\ref{thm:SKhisHH}}

\theoremstyle{definition}
\newtheorem{notation}[theorem]{Notation}
\newtheorem{definition}[theorem]{Definition}
\newtheorem{remark}[theorem]{Remark}

\newcommand{\R}{\ensuremath{\mathbb{R}}}
\newcommand{\Z}{\ensuremath{\mathbb{Z}}}

\newcommand{\C}{\ensuremath{\mathbb{C}}}
\newcommand{\F}{\ensuremath{\mathbb{F}}}

\newcommand{\Id}{\ensuremath{\mbox{\textbb{1}}}}

\newcommand{\boldSigma}{\ensuremath{\mbox{\boldmath $\Sigma$}}}

\newcommand{\cF}{\ensuremath{\mathcal{F}}}
\newcommand{\cG}{\ensuremath{\mathcal{G}}}
\newcommand{\cA}{\ensuremath{\mathcal{A}}}
\newcommand{\cM}{\ensuremath{\mathcal{M}}}
\newcommand{\cN}{\ensuremath{\mathcal{N}}}
\newcommand{\cR}{\ensuremath{\mathcal{R}}}
\newcommand{\cQ}{\ensuremath{\mathcal{Q}}}
\newcommand{\cZ}{\ensuremath{\mathcal{Z}}}

\newcommand{\SKh}{\ensuremath{\mbox{SKh}}}
\newcommand{\CKh}{\ensuremath{\mbox{CKh}}}
\newcommand{\SpanF}{\ensuremath{\mbox{Span}_\F}}
\newcommand{\Braid}{\ensuremath{\mathfrak{B}}}

\begin{document}
\bibliographystyle{plain}

\begin{abstract} In \cite{MR1862802}, Khovanov-Seidel constructed a faithful action of the $(m+1)$--strand braid group, $\Braid_{m+1}$, on the derived category of left modules over a quiver algebra, $A_m$. We interpret the Hochschild homology of the Khovanov-Seidel braid invariant as a direct summand of the {\em sutured Khovanov homology} of the annular braid closure.\end{abstract}
\maketitle

\section{Introduction}
In \cite{MR1740682}, Khovanov constructed an invariant of links in $S^3$ that takes the form of a bigraded abelian group arising as the homology groups of a combinatorially-defined chain complex. The graded Euler characteristic of Khovanov's link homology recovers the Jones polynomial.  

In \cite{MR2113902}, Asaeda-Przytycki-Sikora showed how to extend Khovanov's construction to obtain an invariant of links in any thickened surface with boundary, $F$. When $F$ is an annulus, the topological situation is particularly natural; the thickened annulus, $A \times I$, can be identified with the complement of a standardly-imbedded unknot in $S^3$. As observed by L. Roberts \cite{GT07060741}, the data of the imbedding $A \times I \subset S^3$ endows the Khovanov complex associated to $L \subset S^3$ with a filtration, and the resulting invariant of $L \subset (A \times I \subset S^3)$ is the filtered chain homotopy type of the complex. The induced spectral sequence converges to $\mbox{Kh}(L)$, the Khovanov homology of $L \subset S^3$. 

This filtered complex is particularly well-suited for studying braids up to conjugacy. Explicitly, letting $\Braid_{m+1}$ denote the $(m+1)$--strand braid group, 
we can form the {\em annular closure}, $\widehat{\sigma} \subset A \times I$, of any braid $\sigma \in \Braid_{m+1}$, and the isotopy class of the resulting annular link (hence, the filtered chain homotopy type of the Khovanov complex) is an invariant of the conjugacy class of $\sigma \in \Braid_{m+1}$. 
Indeed, the homology of the associated graded complex, the so-called {\em sutured annular Khovanov homology} of $L \subset A \times I$, 
detects the trivial braid for every $m \in \Z^{\geq 0}$, though it cannot distinguish all pairs of non-conjugate braids (cf. \cite{GT12122222}).  

The main goal of this paper is to establish a relationship between the sutured Khovanov homology of a braid closure and another ``categorified" braid invariant appearing in work of Khovanov and Seidel.  
In \cite{MR1862802}, Khovanov-Seidel 
consider a family of graded associative algebras, $A_m$, each realized as the quotient of a path algebra by a collection of quadratic relations. To each $\sigma \in \Braid_{m+1}$ they associate a differential graded bimodule, $\cM_\sigma$, over $A_m$, well-defined up to homotopy equivalence, and prove that tensoring with $\cM_\sigma$ over $A_m$ yields a well-defined endofunctor on $D^b(A_m-\mbox{mod})$, the bounded derived category of left $A_m$--modules. By relating the action of $\Braid_{m+1}$ on $D^b(A_m-\mbox{mod})$ to its action on the Fukaya category of a certain symplectic manifold (the Milnor fiber of an $A_m$--type singularity), they prove that their categorical action is {\em faithful}, i.e., $\cM_\sigma \cong \cM_{\Id}$ iff $\sigma = \Id$. 
The Khovanov-Seidel construction also gives rise to a braid conjugacy class invariant: the {\em Hochschild homology} of $A_m$ with coefficients in $\mathcal{M}_\sigma$, denoted $HH(A_m, \cM_\sigma)$.

Our main result is a proof that these two braid conjugacy invariants are closely-related; one is a direct summand of the other. To state our result more precisely, we first note that while Khovanov homology is {\em bigraded}, sutured annular Khovanov homology is {\em triply-graded} (the extra {\em filtration} grading appropriately measures ``wrapping" around the annulus). Denote by $\mbox{SKh}(L;f)$ the sutured Khovanov homology in filtration grading $f$.  We prove:

\begin{theoremSKhisHH}
Let $\sigma \in \Braid_{m+1}$,  and $m(\widehat{\sigma}) \subset A \times I$ the mirror of its annular closure.  \[ \SKh\left(m(\widehat{\sigma});m-1\right) \cong HH\left(A_m, \mathcal{M}_\sigma\right).\]
\end{theoremSKhisHH}

Informally, the Hochschild homology of the Khovanov-Seidel bimodule, $\cM_\sigma$, agrees with the ``next-to-top" graded piece of $\mbox{SKh}(m(\widehat{\sigma}))$.\footnote{It is an immediate consequence of the definitions (see Section \ref{sec:SKh}) that if $\sigma \in \Braid_{m+1}$, then $\mbox{SKh}(\widehat{\sigma};f) = 0$ unless $f \in \{-(m+1), -(m-1), \ldots, m-1, m+1\}$.}  Moreover, up to an overall normalization, the bigrading on $\mbox{SKh}(\widehat{\sigma})$ agrees with a natural bigrading on $HH\left(A_m, \mathcal{M}_\sigma\right).$ This is made explicit in the more precise version of Theorem \ref{thm:SKhisHH} stated in Section \ref{sec:MainThm}.

Those readers familiar with strongly-based mapping class bimodules in bordered Heegaard-Floer homology should recognize Theorem \ref{thm:SKhisHH} as a Khovanov homology analogue of (the $1$--moving strand case of) \cite[Thm. 7]{BorderedBimodules}. Indeed, letting $\boldSigma(\widehat{\sigma})$ denote the double-cover of $S^3$ branched over $\widehat{\sigma}$, $p:\boldSigma(\widehat{\sigma}) \rightarrow S^3$ the covering map, and $K_B$ the braid axis, one corollary of Theorem \ref{thm:SKhisHH} is a new proof of the existence of a spectral sequence relating the ``next-to-top" filtration grading of the sutured Khovanov homology of a braid closure to the ``next-to-bottom" Alexander grading of the knot Floer homology of the fibered link, $p^{-1}(K_B)$. See Theorem \ref{thm:OzSzspecseq} for a more precise statement and \cite{GT07060741} (for even index, \cite{AnnularLinks}) for the original proof of this result.

We should also point out that the Khovanov-Seidel algebra is a special case ($k=1$) of a family of algebras $A^{k,n-k}$ (where $n = m+1$), defined independently by Chen-Khovanov \cite{QA0610054} and Stroppel \cite{StroppelAlgebra} (see also \cite{BrundanStroppel}), that give rise to a categorification of the $U_q(\mathfrak{sl}_2)$ Reshetikhin-Turaev invariant of tangles (cf. \cite[Thm. 66]{FrenkelStroppelSussan}). These algebras are subquotient algebras of Khovanov's arc algebra $H^n$ \cite{Tangles} and can be identified with endomorphism algebras of projective generators of certain blocks, $\mathcal{O}_{k,n-k}$, of category $\mathcal{O}$. They also admit a geometric interpretation in terms of convolution algebras of $2$--block Springer fibers \cite{StroppelWebster}. As in the Khovanov-Seidel setting, one obtains, for each $k$, a categorified braid (indeed, tangle) invariant that takes the form of a (derived equivalence class of a) bimodule, $\cM_{\sigma}^{k}$,  over the algebra $A^{k,n-k}$. The following conjecture, generalizing Theorem \ref{thm:SKhisHH}, arose during conversations with Catharina Stroppel:

\begin{conjecture} Let $\sigma \in \Braid_{n}$, and $m(\widehat{\sigma}) \subset A \times I$ the mirror of its annular closure. \[\SKh(m(\widehat{\sigma});n-2k) \cong HH(A^{k,n-k}, \cM_\sigma^k)\]
\end{conjecture}

Establishing this conjecture may be useful for computational purposes, since modifying existing Khovanov homology programs to compute sutured Khovanov homology should be straightforward.

{\bf Acknowledgments:} We thank John Baldwin, Tony Licata, and Catharina Stroppel for a number of interesting conversations, as well as Robert Lipshitz for explaining how to obtain Proposition \ref{prop:HHbordsut}, generalizing \cite[Thm. 7]{BorderedBimodules}, from work of Rumen Zarev \cite{GT09081106}. We would also like to thank MSRI and the organizers of the semester-long program on Homology Theories of Knots and Links, during which a portion of this work was completed.

\section{Topological preliminaries}
As in \cite[Sec. 3.1]{QuiverAlgebras}, let $D_m \subset \C$ denote the standard unit disk equipped with a set, \[\Delta := \left\{-1 +\frac{2(j+1)}{m+2} \in D_m \subset \C \,\,\vline\,\, j= 0, \ldots, m\right\},\] of $m+1$ points equally spaced along the real axis.  $I = [0,1]$ is the closed, positively-oriented unit interval.

Recall that the $(m+1)$--strand  braid group, $\Braid_{m+1}$, is the set of equivalence classes of properly-imbedded smooth $1$--manifolds $\sigma \subset D_m \times I$ satisfying the properties:
\begin{enumerate}
	\item  $\partial \sigma = (\Delta \times \{0\}) \cup (\Delta \times \{1\})$,
	\item $|\sigma \cap D_m \times \{t\}| = m+1$ for all $t \in I$.
\end{enumerate}
Two such $1$--manifolds $\sigma$ and $\sigma'$ are considered {\em equivalent} ({\em braid isotopic}) if there exists a smooth isotopy from $\sigma$ to $\sigma'$ through $1$--manifolds satisfying Properties (1) and (2). 

$\Braid_{m+1}$ forms a group, with composition given by stacking (bottom to top) and vertical rescaling.  We shall make frequent use of Artin's well-known presentation:
\[\Braid_{m+1} := \left\langle \sigma_1, \ldots, \sigma_m \,\,\vline\,\, \begin{array}{cl} \sigma_i\sigma_j = \sigma_j\sigma_i & \mbox{if $|i-j| = 1$,}\\
\sigma_i\sigma_j\sigma_i = \sigma_j\sigma_i\sigma_j & \mbox{if $|i-j| \geq 2$.}\end{array}\right\rangle.\] Our convention will be that Artin words read left to right denote braids formed by stacking the elementary Artin generators bottom to top.

Associated to a factorization, $\sigma = \sigma_{i_1}^\pm \cdots \sigma_{i_k}^\pm$, of $\sigma \in \Braid_{m+1}$ as a product of elementary Artin generators, one constructs a diagram of $\sigma$ in $\R \times I$. Any crossing of such a diagram can be resolved in one of two ways as indicated in Figure \ref{fig:Resolutions}, and any complete resolution is an $(m+1)$--strand {\em Temperley-Lieb (TL) diagram}, i.e., a smooth, properly-imbedded $1$--manifold $U \subset \R \times I$ with $\partial U = (\Delta \times \{0\}) \cup (\Delta \times \{1\}).$ As above, TL diagrams $U$ and $U'$ are considered equivalent if $U$ is isotopic to $U'$ through TL diagrams, and there is a multiplicative structure on TL diagrams, obtained by stacking and vertically rescaling. The product diagram, $\mbox{Id} := \Delta \times I,$ is a two-sided identity with respect to this multiplication.  Any TL diagram is isotopic to a product of the elementary TL diagrams $U_1, \ldots, U_m$ pictured in Figure \ref{fig:TLgenerators}, and two {\em TL words} in these generators represent isotopic TL diagrams iff one can be obtained from the other by applying {\em Jones relations}:
\begin{eqnarray}
\label{eqn:trivcirc} U_i^2 &=& \bigcirc \amalg U_i \,\,\mbox{ for all $i \in \{1, \ldots, m\}$}\\
\label{eqn:kinkright}U_iU_{i+1}U_i &=& U_i \,\,\mbox{ for all $i \in \{1, \ldots, m-1\}$}\\
\label{eqn:kinkleft}U_iU_{i-1}U_i &=& U_i \,\,\mbox{ for all $i \in \{2, \ldots, m\}$}
\end{eqnarray}

The ``$\bigcirc$" of \eqref{eqn:trivcirc} represents a closed circle positioned between the cup and cap of $U_i$.

\begin{figure}
\begin{center}
\resizebox{1in}{!}{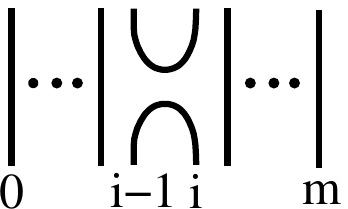}
\end{center}
\caption{The elementary TL generator, $U_i$}
\label{fig:TLgenerators}
\end{figure}

\section{Algebraic Preliminaries}

We refer the reader to \cite{MR1854636}, \cite{MR2258042}, \cite{SeidelBook}, \cite{kenji}, \cite{GT08100687} for standard background on $A_\infty$ algebras and modules. Terminology and notation are as in \cite{QuiverAlgebras}. All algebras/modules we consider are over $\F := \Z/2\Z$.

\begin{notation} \label{not:grading}
Given a bigraded vector space  \[V = \bigoplus_{i,j \in \Z} V_{(i,j)},\] (e.g., a differential graded module) and $k_1, k_2 \in \Z$,  $V[k_1]\{k_2\}$ will denote the bigraded vector space whose first (homological) grading has been shifted {\em down} by $k_1$ and whose second (internal) grading has been shifted {\em up} by $k_2$.\footnote{The difference in shift conventions for homological versus internal gradings is unfortunate, but standard in the literature. In particular, they coincide with those in \cite{MR1862802}, to which we frequently refer.} Explicitly, \[\left(V[k_1]\{k_2\}\right)_{(i-k_1,j+k_2)} := V_{(i,j)}.\]
\end{notation}

\begin{notation} \label{not:gradingshift}
If $V$ is a bigraded vector space, $k \in \Z^{\geq 0}$, and $W := \F_{(0,-1)} \oplus \F_{(0,1)},$ then we will use the notation $\bigcirc_{k}V$ to denote the bigraded vector space $V \otimes W^{\otimes{k}}$. 
\end{notation}

\begin{notation} \label{not:derivedcat}
Let $A$ and $B$ be $A_\infty$ algebras. 
We will denote by $D_\infty(A)$ (resp., by $D_\infty(A^{op})$, $D_\infty(A-B^{op})$) the category whose objects are homologically unital $A_\infty$ {\em left--} (resp., {\em right--}, {\em bi-}) modules over $A$ (resp., over $A$, over $A$--$B$) and whose morphisms are $A_\infty$ homotopy classes of $A_\infty$ morphisms.
\end{notation}

\begin{definition} Let $A,B,$ and $C$ be $A_\infty$ algebras, $M \in D_\infty(A-B^{op})$, and $N \in D_\infty(B -C^{op}).$ The derived $A_\infty$ tensor product, $M \widetilde{\otimes}_B N$, is the well-defined object of $D_\infty(A - C^{op})$ represented by the $A_\infty$ bimodule over $A$--$C$ with underlying vector space \[\bigoplus_{k=0}^\infty M \otimes B^{\otimes k} \otimes N,\] whose differential ($m_{(0|1|0)}$ structure map) is given in terms of the structure maps $m^M$, $m^B$, and $m^N$ by 
\[m_{(0|1|0)}^{M \widetilde{\otimes}_B N} := \sum_{i_1\leq k} m_{(0|1|i_1)}^M \otimes \mbox{Id}^{\otimes k-i_1+1} + \sum_{i_2 \leq k} \mbox{Id}^{\otimes k-i_2+1} \otimes m_{(i_2|1|0)}^N\]
\[+\sum_{1 \leq j_1 < j_2 \leq k} \mbox{Id}^{\otimes j_1} \otimes m_{j_2-j_1}^B \otimes \mbox{Id}^{\otimes k-j_2 + 2}\] and whose higher ($i_1>0$ or $i_2>0$) $A_\infty$ structure maps,
\[m_{(i_1|1|i_2)}^{M \widetilde{\otimes}_B N} : A^{\otimes i_1} \otimes \left(M \otimes B^{\otimes k} \otimes N\right) \otimes C^{\otimes i_2} \rightarrow \bigoplus_{\substack{0 \leq j \leq k}} M \otimes B^{\otimes k-j} \otimes N,\] are given by 
\[m_{(i_1|1|i_2)}^{M \widetilde{\otimes}_B N} := \left\{\begin{array}{cl}
\sum_{0 \leq j \leq k} m^M_{(i_1|1|j)} \otimes \mbox{Id}^{\otimes k - j + 1} & \mbox{if $i_2 = 0$,}\\
\sum_{0 \leq j \leq k} \mbox{Id}^{\otimes k - j + 1} \otimes m^M_{(j|1|i_2)} & \mbox{if $i_1 = 0$,}\\
0 & \mbox{otherwise.}\end{array}\right.\]
\end{definition}

\begin{definition} \label{defn:HochHomAinfty} Let $A$ be an $A_\infty$ algebra and $M \in D_\infty(A - A^{op})$. The {\em Hochschild homology} of $A$ with coefficients in $M$, denoted $HH(A,M)$, is the derived $A_\infty$ self-tensor product of $M$. Explicitly, it is the homology of the chain complex with underlying vector space:
\[\bigoplus_{k=0}^m A^{\otimes k} \otimes M\] and differential \[\partial(a_1 \otimes \ldots \otimes a_k \otimes m) :=\]
\[ \sum_{\substack{0 \leq i_1, i_2\\ i_1+i_2 \leq k}} a_{i_2+1} \otimes \ldots \otimes a_{k-i_1} \otimes m_{(i_1|1|i_2)}^M(a_{k-i_1+1} \otimes \ldots \otimes a_k \otimes m \otimes a_1 \otimes \ldots \otimes a_{i_2})\]
\[+ \sum_{1 \leq j_1 < j_2 \leq k} a_1 \otimes \ldots \otimes m^A_{j_2-j_1+1}(a_{j_1} \otimes \ldots \otimes a_{j_2}) \otimes a_{j_2+1} \otimes \ldots \otimes a_k \otimes m\]
\end{definition}

When $A$ is an ordinary associative algebra and $M$ is a dg bimodule, the above definition agrees with the ``classical" definition of Hochschild homology:

\begin{definition} \label{defn:HochHomClassical}
Let $A$ be an associative algebra and $M$ a dg bimodule over $A$. Noting that $M$ (resp., $A$) can be viewed as a left (resp., right) module over $A^e := A \otimes A^{op}$, the {\em $n$th (classical) Hochschild homology } of $A$ with coefficients in $M$ is \[HH_n(A,M) := \mbox{Tor}_n^{A^e}(A,M).\]
\end{definition}

In particular, to compute $HH(A,M)$ one chooses a resolution \[\mathcal{R}(A) \rightarrow A\] of $A$ by projective right $A^e$ modules, and \[HH_n(M) := H_n(\cR(A) \otimes_{A^e} M).\] Standard arguments in homological algebra imply that the homology is independent of the chosen resolution. When $\cR(A)$ is the bar resolution, we obtain:

\[HH_n(A,M) := H_n(\xymatrix{\cdots \ar[r]^-{\partial_4} & A^{\otimes 3} \otimes M \ar[r]^-{\partial_3} & A^{\otimes 2} \otimes M \ar[r]^-{\partial_2} & A \otimes M \ar[r]^-{\partial_1} & M}),\] where 
\[\partial_i(a_1 \otimes \ldots \otimes a_i \otimes m) := (a_1 \otimes \ldots \otimes a_i \otimes \partial^Mm) +\]\[(a_1a_2 \otimes a_3 \otimes \ldots a_i \otimes m) + (a_1 \otimes a_2a_3 \otimes \ldots a_i \otimes m) + \ldots + (a_2 \otimes a_3 \otimes \ldots \otimes ma_1),\] and $\partial^M$ is the internal differential on $M$.

Since the higher structure maps of $A$ and $M$ are assumed to be zero, the complex described in Definition \ref{defn:HochHomAinfty} agrees with this one. Accordingly, we shall use the definitions interchangeably when discussing the Hochschild homology of an ordinary associative algebra with coefficients in a dg bimodule.

\begin{lemma} \label{lem:HHTensor} Let $A$ and $B$ be $A_\infty$ algebras. We have \[HH(A, M \widetilde{\otimes}_B N) \cong HH(B, N \widetilde{\otimes}_A M)\] for any $A_\infty$ $A$--$B$ bimodule $M$ and $A_\infty$ $B$--$A$ bimodule $N$.
\end{lemma}

\begin{proof} The Hochschild complexes described in Definition \ref{defn:HochHomAinfty} are chain isomorphic, via the ($\F$--linear extension of) the canonical identification \[(a_1 \otimes \ldots \otimes a_k) \otimes m \otimes (b_1 \otimes \ldots \otimes b_{\ell}) \otimes n \longleftrightarrow (b_1 \otimes \ldots \otimes b_\ell) \otimes n \otimes (a_1 \otimes \ldots \otimes a_k) \otimes m.\]
\end{proof}

\begin{corollary} \label{cor:Morita} Let $A$ and $B$ be $A_\infty$ algebras. Suppose that there exists an $A_\infty$ $A$--$B$ bimodule ${}_AP_B$ and an $A_\infty$ $B$--$A$ bimodule ${}_BP_A$ satisfying the property that 
\begin{eqnarray*}
	{}_AP_B \,\,\widetilde{\otimes}_B\,\, {}_BP_A \cong A &\in & D_\infty (A-A^{op}).
\end{eqnarray*} 
Then for any $M \in D_\infty(A-A^{op})$, we have \[HH(A,M) \cong HH(B,{}_BP_A \,\,\widetilde{\otimes}_A \,\, M \,\,\widetilde{\otimes}_A \,\,{}_AP_B).\]
\end{corollary}

\begin{proof}
By Lemma \ref{lem:HHTensor} and the equivalence ${}_AP_B \,\,\widetilde{\otimes}_B \,\,{}_BP_A \cong A$, we have 
\begin{eqnarray*}
	HH(B, {}_BP_A \,\,\widetilde{\otimes}_A \,\, M \,\,\widetilde{\otimes}_A \,\,{}_AP_B) &\cong& HH(A, M \,\,\widetilde{\otimes}_A \,\,{}_AP_B \,\,\widetilde{\otimes}_B \,\,{}_BP_A)\\
	&\cong& HH(A,M),
\end{eqnarray*}
as desired.
\end{proof}

We will also make use of the following non-derived version of a self-tensor product for dg bimodules over an associative algebra.

\begin{definition} \label{defn:coinvariant}
Let $A$ be an associative algebra (i.e., a dg algebra supported in a single homological grading) and let \[\cM = \left(\xymatrix{\cdots \ar[r] & \cM_{n} \ar[r]^-{\partial^{n}} & \cM_{n+1} \ar[r] & \cdots}\right)\] be a dg $A$--bimodule. 

Then the {\em coinvariant quotient complex} of $\cM$ is the complex:
\[\cQ(\cM) :=  \left(\xymatrix{\cdots \ar[r] & \cQ(\cM_{n}) \ar[r]^-{\cQ(\partial^{n})} & \cQ(\cM_{n+1}) \ar[r] & \cdots}\right),\]
where \[\cQ(\cM_{n}) := \frac{\cM_{n}}{\langle am - ma\,\,|\,\, a \in A, m \in \cM_{n}\rangle},\] and $\cQ(\partial^n)$ is the induced differential on the quotient (well-defined, since $\partial^n$ is $A$--linear). 
\end{definition}

We have the following analogue of Lemma \ref{lem:HHTensor}:

\begin{lemma} \label{lem:QCyclic}
Let $A, B$ be associative algebras, and let $\cM$ (resp., $\cN$) be dg bimodules over $A-B$ (resp., over $B-A$). The canonical $\F$--linear map \[\Psi: \cQ(\cM \otimes_{B} \cN) \rightarrow \cQ(\cN \otimes_{A} \cM)\] sending $[m \otimes n] \in \cQ(\cM \otimes_{A} \cN)$ to $[n \otimes m] \in \cQ(\cN \otimes_B \cM)$ is a chain isomorphism.
\end{lemma}

\begin{proof}
$\Psi$ is well-defined, since for any $a \in A, b \in B, m \in \cM, n \in \cN$ we have
\[\Psi[mb \otimes n] = [n \otimes mb] = [bn \otimes m] = \Psi[m \otimes bn]\]
and
\[\Psi[am \otimes n] = [n \otimes am] = [na \otimes m] = \Psi[m \otimes na].\] Verifying that $\Psi$ is a chain map (with canonical inverse)
is similarly routine.
\end{proof}

\section{Sutured annular Khovanov homology and Khovanov-Seidel bimodules}

We begin by reviewing the definition of sutured annular Khovanov homology in \cite{MR2113902} (see also \cite{GT07060741}, \cite{AnnularLinks}) as well as the construction of Khovanov-Seidel in  \cite{MR1862802}.

\subsection{Sutured annular Khovanov homology} \label{sec:SKh}
Sutured annular Khovanov homology is an invariant of links in a solid torus (equipped with a fixed identification as a thickened annulus), defined as follows.

Let $A$ denote an oriented annulus and $I = [0,1]$ the closed, positively oriented, unit interval. Any link $L \subset A \times I$ admits a diagram on $A \times \{\frac{1}{2}\}$, which can equivalently be viewed as a diagram on $S^2 - \{O,X\},$ where $O,X$ are two distinguished points (the north and south poles, e.g.). See Figure \ref{fig:AnnularDiag}.

\begin{figure}
\begin{center}
\resizebox{4in}{!}{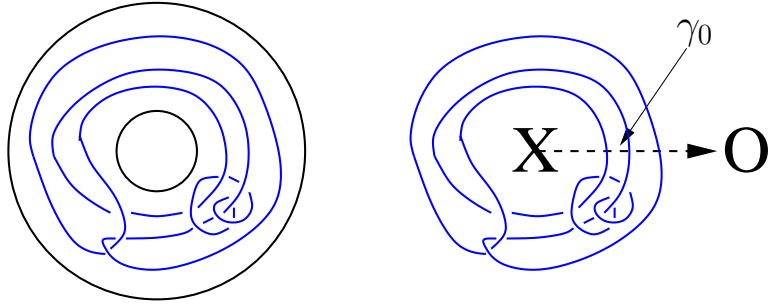}
\end{center}
\caption{Pictured above are two diagrams of an annular link, $L \subset A \times I$, one on $A \times \left\{\frac{1}{2}\right\}$ (left) and one on $S^2$ equipped with two distinguished points,  $O$ and $X$ (right). The oriented arc, $\gamma_0$, specifies the ``$f$" (filtration) grading on the Khovanov complex associated to the annular link.}
\label{fig:AnnularDiag}
\end{figure}

By forgetting the data of the $X$ basepoint, one obtains a diagram of $L$ on $D := S^2 - N(O)$. At this point, one can construct the (bigraded) Khovanov complex as normal by 
ordering the $k$ crossings of the diagram of $L$ and assigning to the diagram its $k$--dimensional {\em cube of resolutions}, $[-1,1]^k$, whose vertices, $\vec{v} = (v_1, \ldots, v_k) \in \{-1,1\}^k$, correspond to complete resolutions as in Figure \ref{fig:Resolutions}.

\begin{figure}
\begin{center}
\resizebox{4in}{!}{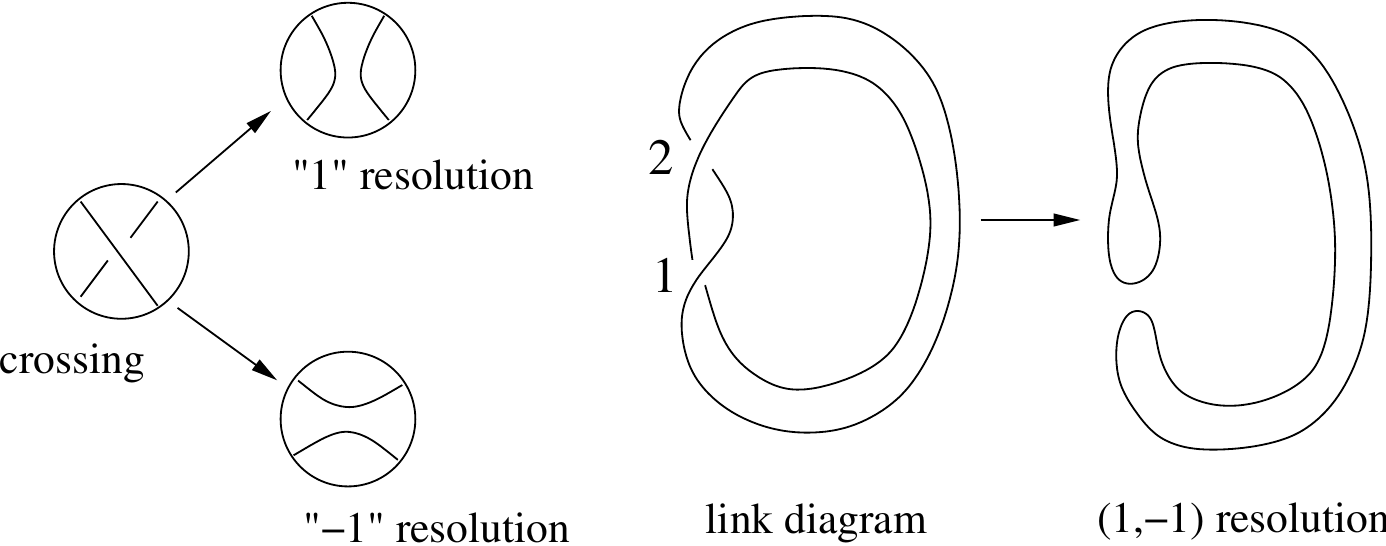}
\end{center}
\caption{If the $k$ crossings of a link diagram have been ordered, we can associate to each vertex, $(v_1, \ldots, v_k) \in \{-1,1\}^k$, of the $k$--cube a complete resolution of the diagram as indicated.}
\label{fig:Resolutions}
\end{figure}

If we now remember the data of the $X$ basepoint, we may choose an oriented arc, $\gamma_0$, from $X$ to $O$ missing all crossings of the diagram. As described in \cite[Sec. 4.2]{JacoFest}, the generators of the Khovanov complex are in $1:1$ correspondence with {\em enhanced} (i.e., {\em oriented}) resolutions, so from the data of the oriented arc we obtain an extra ``$f$" (filtration) grading on the complex from the algebraic intersection number of the oriented arc with the oriented resolution. Roberts proves (\cite[Lem. 1]{GT07060741}), that the Khovanov differential is non-increasing in this extra grading, so one obtains a filtration of the Khovanov complex, $\CKh(L)$:
\[ 0 \subseteq \ldots \subseteq \cF_{n-1} \subseteq \cF_{n} \subseteq \cF_{n+1} \subseteq \ldots \subseteq \CKh(L),\] whose $n$th subcomplex is given by:
\[\mathcal{F}_n := \oplus_{f \leq n} \CKh(L; f).\] The sutured annular Khovanov homology of $L$ is then defined to be the homology of the associated graded complex:
\[\SKh(L) := \bigoplus_{n \in \Z} \SKh(L; n) := \bigoplus_{n \in \Z} H_*\left(\frac{\mathcal{F}_n}{\mathcal{F}_{n-1}}\right).\]

A concrete description of the sutured annular chain complex is given in \cite[Sec. 2]{GT07060741}. See also \cite{AnnularLinks} and \cite{JacoFest}.

In the present work, we will be interested in the special case when $L$ is the {\em annular braid closure}, $\widehat{\sigma}$, of a braid, $\sigma$. Precisely, if $\sigma \subset D_m \times I$ is a braid, then we obtain $\widehat{\sigma} \subset A \times I$, by gluing $D_m \times \{0\}$ to $D_m \times \{1\}$ by the identity diffeomorphism and choosing the identification $D_m \times S^1 \approx A \times I$ that on each radial slice agrees with the homeomorphism of Figure \ref{fig:Homeo}. 

\begin{figure}
\begin{center}
\resizebox{1in}{!}{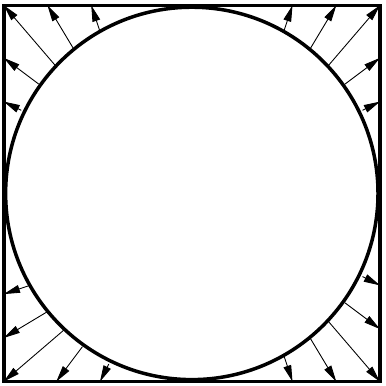}
\end{center}
\caption{A radial slice of the homeomorphism $D_m \times S^1 \approx A \times I$}
\label{fig:Homeo}
\end{figure}

Note that if we view $S^3$ as $\R^3 \cup \infty$, there is a standard imbedding of the annular closure of $\sigma$ as a link in the complement of the braid axis, \[K_B := \{(r,\theta,z)\,\,|\,\,r=0\} \cup \infty \subset S^3,\] via the homeomorphism identifying $A \times I$ with $S^3 - N(K_B)$:
\[ A \times I = \{(r,\theta,z)\,\,\vline\,\,r \in [1,2], \theta \in [0, 2\pi), z \in [0,1]\} \subset S^3.\]

In preparation for establishing a correspondence with the Khovanov-Seidel bimodules described in the next section, we now give an explicit description of the sutured annular chain complex of a factorized annular braid closure. Order the $k$ crossings of $\sigma = \sigma_{i_1}^\pm \cdots \sigma_{i_k}^\pm$ from bottom to top, as read from left to right (where $\sigma_i^\pm$ is as pictured in \ref{fig:Braidgen}), and assign the same ordering to the crossings of the annular diagram of $\widehat{\sigma} \subset A \times I$. Endow $\widehat{\sigma}$ with the braid orientation.

Now consider the $k$--dimensional cube, $[-1,1]^k$, whose vertices are indexed by $k$--tuples $\vec{v} = (v_1, \ldots, v_k) \in \{-1,1\}^k$. Letting $\epsilon_j(\sigma_{i_1}^\pm \cdots \sigma_{i_k}^\pm) \in \{-1,1\}$ denote the exponent on the $j$th term in the product, we can now identify the vertex $\vec{v}$ of the $k$--cube with the annular closure of the TL diagram, $N_{i_1} \cdots N_{i_k}$, where:
\begin{equation}\label{eqn:Khorientconv}
N_{i_j} := \left\{\begin{array}{cl}
U_{i_j} & \mbox{if $v_j\epsilon_j = 1$, and}\\
\mbox{Id} & \mbox{if $v_j\epsilon_j = -1$,}
\end{array}\right.
\end{equation}

Moreover, if there is a directed edge from a vertex $\vec{v}$ to a vertex $\vec{w}$ in the $k$--cube then $\vec{v},\vec{w}$ agree except at a single entry at (say) the $j$--th position, where $v_j=-1$ and $w_j = +1$. The map assigned to this edge is one of the split/merge maps described in \cite[Sec. 2]{GT07060741}, depending on the topological configuration of the splitting/merging circles.

Let $n_+$ (resp., $n_-$) denote the number of positive (resp., negative) crossings of $\widehat{\sigma}$. The (h)omological, (q)uantum, and (f)iltration gradings of a sutured annular Khovanov generator, ${\bf x}$, is given in terms of its associated oriented resolution (enhanced Kauffman state), $\mathbb{S}({\bf x})$, by:
\begin{eqnarray*}
	h({\bf x}) &=& -n_- + \frac{1}{2}\sum_{i=1}^k (1 + v_i({\bf x}))\\
	q({\bf x}) &=& (n_+ - 2n_-) + \mbox{CCW}(\mathbb{S}({\bf x})) - \mbox{CW}(\mathbb{S}({\bf x})) + h({\bf x})\\
	f ({\bf x}) &=& [\gamma_0]\cdot [{\mathbb{S}(\bf x})],
\end{eqnarray*}
where ``$\cdot$" denotes algebraic intersection number, and $\mbox{CCW}(\mathbb{S}({\bf x}))$ (resp., $\mbox{CW}(\mathbb{S}({\bf x}))$) denotes the number of components of $\mathbb{S}({\bf x})$ oriented counter-clockwise (resp., clockwise).

Note that if $\sigma \in \Braid_{m+1}$, then $\SKh(\widehat{\sigma};k) = 0$ for $k \not\in \{-(m+1), -(m-1), \ldots, m-1, m+1\}$ (cf. \cite[Prop. 7.1]{QuiverAlgebras}), so we will sometimes refer to $\SKh(\widehat{\sigma};m-1)$ as the {\em next-to-top} filtration level of $\SKh(\widehat{\sigma})$. 
  
\subsection{Khovanov-Seidel bimodules} \label{sec:KhS}
Let $\Gamma_m$ be the oriented graph (quiver) with 
\begin{itemize}
	\item vertices labeled $0, \ldots, m$ and 
	\item a pair of oppositely--oriented edges connecting each pair of adjacent vertices, as in Figure \ref{fig:QuiverDm}.
\end{itemize}

\begin{figure}
\begin{center}
\resizebox{1.5in}{!}{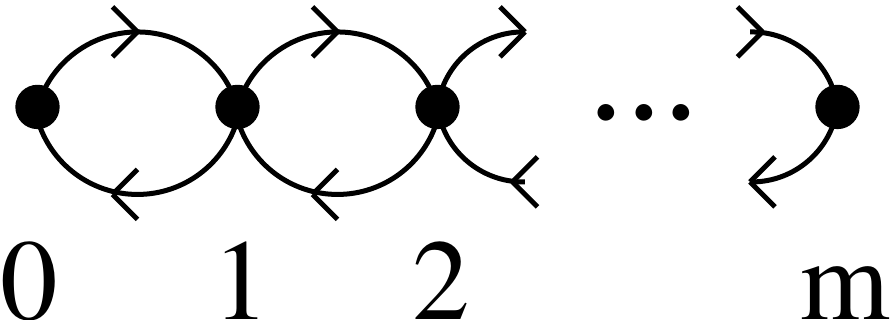}
\end{center}
\caption{The oriented graph $\Gamma_m$ used in the definition of the algebra $A_m$}
\label{fig:QuiverDm}
\end{figure}

Recall that, given any oriented graph $\Gamma$, one defines its path algebra as the algebra whose underlying vector space is freely generated by the set of all finite-length paths in $\Gamma$, and multiplication is given by concatenation (the product of two non-composable paths is $0$).  

The algebra $A_m$ is then defined as a quotient of the path algebra of $\Gamma_m$ by the collection of relations \[(i-1|i|i+1) = (i+1|i|i-1)=0, \hskip 10pt (i|i+1|i)=(i|i-1|i), \hskip 10pt (0|1|0)=0\] for each $1 \leq i \leq m-1$.  In the above, following  \cite{MR1862802}, we have labeled each path in $\Gamma_m$ by the complete ordered tuple of vertices it traverses.  So, for instance, $(i|i+1|i)$ denotes the path that starts at vertex $i$, moves right to $i+1$, then returns to $i$.  

The path algebra of $\Gamma_m$ is further endowed with an internal grading by {\em negative} path length. As the above relations are homogenous with respect to this grading, it descends to the quotient, $A_m$.\footnote{The reader should be warned that we are using a different internal grading than the one considered in \cite{MR1862802} and \cite{QuiverAlgebras}. It is this {\em negative path length} grading and not Khovanov-Seidel's {\em steps-to-the-left} grading that is most directly related to Khovanov homology. Note also that $A_m$ is Koszul with respect to the {\em positive} path length grading, cf. \cite{StroppelAlgebra}.}

The collection, $\{(i)|i \in 0,\ldots,m\}$, of constant paths are mutually orthogonal idempotents whose sum,  $1 = \sum_{i=0}^m (i)$, is the identity in $A_m$.  There are corresponding decompositions of $A_m$ as a direct sum of projective left-modules $A_m = \bigoplus_{i=0}^m A_m(i)$ (resp., projective right-modules $A_m = \bigoplus_{i=0}^m (i)A_m$).  As in  \cite{MR1862802}, we denote $A_m(i)$ (resp., $(i)A_m$) by $P_i$ (resp., ${}_i{P}$).  Note that $P_i$ (resp., ${}_i{P}$) is the set of all paths ending at $i$ (resp., beginning at $i$).

Khovanov-Seidel go on to construct a braid group action on $D^b(A_m)$, the bounded derived category of left $A_m$--modules, by associating to each elementary braid word $\sigma_i^{\pm 1}$ (pictured in Figure \ref{fig:Braidgen}) a differential bimodule $\mathcal{M}_{{\sigma_i^\pm}}$ and to each braid, $\sigma := {\sigma_{i_1}}^\pm \cdots {\sigma_{i_k}}^\pm$, decomposed as a product of elementary braid words, the differential bimodule \[\mathcal{M}_\sigma = \mathcal{M}_{\sigma_{i_1}^\pm} \otimes_{A_m} \ldots \otimes_{A_m} \mathcal{M}_{\sigma_{i_k}^\pm}.\]

\begin{figure}
\begin{center}
\resizebox{2in}{!}{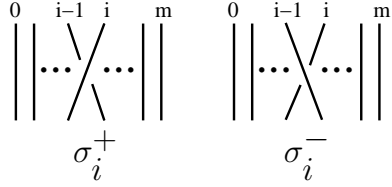}
\end{center}
\caption{The elementary Artin braid group generators}
\label{fig:Braidgen}
\end{figure}

Specifically, they associate to $\sigma_i^+$ the differential bimodule \[\mathcal{M}_{\sigma_i^+} := \xymatrix{P_i \otimes {}_i{P}\{1\}  \ar[r]^-{\beta_i} & A_m\{1\}},\] where $\beta_i$ is the $A_m$--bimodule map induced by the assignment $\beta_i((i)\otimes(i)) = (i)$, and to $\sigma_i^{-}$ the differential bimodule \[\mathcal{M}_{\sigma_i^-} := \xymatrix{A_m \ar[r]^-{\gamma_i} & P_i \otimes {}_i{P}\{2\}},\] where \[\gamma_i(1) = (i-1|i) \otimes (i|i-1) + (i+1|i)\otimes (i|i+1) + (i) \otimes (i|i-1|i) + (i|i-1|i) \otimes (i).\]

Suppose now that we begin with $\sigma \in \Braid_{m+1}$, along with a decomposition $\sigma = \sigma_{i_1}^\pm \cdots \sigma_{i_k}^\pm$. Noting that the Khovanov-Seidel braid module $\mathcal{M}_\sigma$ is a tensor product (over $A_m$) of two-term complexes, each of which is a mapping cone of an $A_m$--module map between the bimodule ($P_{i} \otimes {}_{i}P$) assigned to the $i$--th elementary Temperley-Lieb (TL) object and the bimodule ($A_m$) assigned to the trivial TL object, we have a ``cube of resolutions" description of $\mathcal{M}_\sigma$ in terms of the given decomposition of $\sigma$ as follows which is similar in structure to the cube of resolutions description of the sutured annular Khovanov homology of $\widehat{\sigma} \subset A \times I$.

Order the $k$ crossings of the braid diagram $\sigma = \sigma_{i_1}^\pm \cdots \sigma_{i_k}^\pm$ in the direction of the braid orientation (by convention, from bottom to top as read from left to right). Now consider the $k$--dimensional cube, $[-1,1]^k$, whose vertices are indexed by $k$--tuples $\vec{v} = (v_1, \ldots, v_k) \in \{-1,1\}^k$. Letting $\epsilon_j(\sigma_{i_1}^\pm \cdots \sigma_{i_k}^\pm) \in \{-1,1\}$ denote the exponent on the $j$th term in the product, we can now (modulo bigrading shifts, specified in Lemma \ref{lem:gr}) identify the vertex $\vec{v}$ of the $k$--cube with the term, $\mathcal{N}_{i_1} \otimes_{A_m} \ldots \otimes_{A_m} \mathcal{N}_{i_k}$, of the complex $\mathcal{M}_\sigma$, where:
\begin{equation} \label{eqn:KhSorientconv} \mathcal{N}_{i_j} := \left\{\begin{array}{cl}
A_m & \mbox{if $v_j\epsilon_j = 1$, and}\\
P_{i_j} \otimes {}_{i_j}P & \mbox{if $v_j\epsilon_j = -1$,}
\end{array}\right.
\end{equation}

Moreover, if there is a directed edge from a vertex $\vec{v}$ to a vertex $\vec{w}$ in the $k$--cube then $\vec{v},\vec{w}$ agree except at a single entry at (say) the $j$--th position. The map assigned to this edge is then \[\mbox{Id} \otimes \ldots \otimes \left\{\begin{array}{c} \beta_{i_j}\\ \gamma_{i_j}\end{array}\right\} \otimes \ldots \otimes \mbox{Id}\] according to whether $\epsilon_j = \left\{\begin{array}{c} +1\\-1\end{array}\right\}.$  

As the homological and quantum grading shifts of generators at vertices in this ``cube of resolutions" complex for $\cM_{\sigma} = \cM_{\sigma_{i_1}^\pm \cdots \sigma_{i_k}^\pm}$ are somewhat cumbersome, we take a moment to record them here.

\begin{lemma} \label{lem:gr}
The bigraded Khovanov-Seidel bimodule associated to the vertex $\vec{v} \in \{-1,1\}^k$ in the above cube of resolutions is \[\left(\cN_{i_1} \otimes_{A_m} \ldots \otimes_{A_m} \cN_{i_k}\right)[-\vec{v}_h]\{\vec{v}_q\},\] where
\[\vec{v}_h := \sum_{j=1}^k \frac{1}{2}(v_j + 1),\] and
\[\vec{v}_q = \vec{v}_h + \sum_{j=1}^k \frac{1}{2}(1-v_j\epsilon_j).\]
\end{lemma}

\begin{proof} By definition, $\vec{v}_h$ is the number of $1$'s in the $k$--tuple $\vec{v} \in \{-1,1\}^k$, so the first statement follows.

To compute $\vec{v}_q$, note that we get an overall $+1$ shift for every $j \in \{1, \ldots, n\}$ satisfying $\epsilon_j = 1$ and an extra $+2$ shift for each $j \in \{1, \ldots, k\}$ satisfying $\epsilon_j = -1$ and $v_j = 1$. It follows that
\begin{eqnarray*}
	\vec{v}_q &=& \sum_{j=1}^k \frac{1}{2}(\epsilon_j + 1) + \sum_{j=1}^k -\frac{1}{2} (\epsilon_j - 1)(v_j+1)\\
	&=& \vec{v}_h + \sum_{j=1}^k \frac{1}{2}(1 - v_j \epsilon_j),
\end{eqnarray*}
as desired.
\end{proof}

Note that $\sum_{j=1}^k \frac{1}{2}(1- v_j\epsilon_j)$ is the length of the TL word in elementary TL generators corresponding to the vertex $\vec{v}$. Equivalently, it is the number of terms of the form $(P_{i_j} \otimes {}_{i_j}P)$ in the corresponding tensor product $\cN_{i_1} \otimes_{A_m} \ldots \otimes_{A_m} \cN_{i_k}$.

\section{Main Theorem} \label{sec:MainThm}
In the following, 
\begin{itemize}
	\item $\cM_{\sigma}$ denotes the Khovanov-Seidel bimodule associated to a braid, $\sigma = \sigma_{i_1}^\pm \cdots \sigma_{i_k}^\pm \in \Braid_{m+1}$ as described in Section \ref{sec:KhS}, 
	\item $\SKh(m(\widehat{\sigma}); m-1)$ denotes the ``next-to-top" (f)iltration grading of the sutured Khovanov homology of $m(\widehat{\sigma}) \subset A \times I$, the {\em mirror} (all crossings reversed) of the braid closure, $\widehat{\sigma}$, in $A \times I$, as in Section \ref{sec:SKh}.
\end{itemize}

\begin{theorem} \label{thm:SKhisHH} Let $\sigma \in \Braid_{m+1}$ and $m(\widehat{\sigma}) \subset A \times I$ the mirror of its annular closure.  \[ \SKh\left(m(\widehat{\sigma});m-1\right) \cong HH\left(A_m, \mathcal{M}_\sigma\right)[n_-]\{(m-1) + n_+-2n_-\}\] as {\em bigraded} vector spaces.
\end{theorem}

\begin{remark} The {\em mirror} of $\widehat{\sigma}$ appears on the left-hand side of the equivalence above because of differing braid conventions in \cite{MR1740682} and \cite{MR1862802}, cf. \eqref{eqn:Khorientconv} and \eqref{eqn:KhSorientconv}.
\end{remark}
\begin{proof}
We prove, in Proposition \ref{prop:SKhiscoinvar}, that \[\SKh\left(m(\widehat{\sigma}); m-1\right) \cong H_*(\cQ(\mathcal{M}_\sigma))[n_-]\{(m-1)+n_+-2n_-\},\] where $H_*(\cQ(\cM_\sigma))$ is the homology of the coinvariant quotient module\footnote{It is proven in \cite{MR1862802} that the homotopy equivalence class of $\cM_\sigma$ does not depend on the chosen factorization of $\sigma$ as a product of elementary braids. Since any homotopy equivalence descends to the coinvariant quotient, we are justified in suppressing the factorization from the notation.} (Definition \ref{defn:coinvariant}) of \[\cM_\sigma := \mathcal{M}_{\sigma_{i_1}^\pm \cdots \sigma_{i_k}^\pm}.\]  But Proposition \ref{prop:HHiscoinvar} tells us that \[H_*(\cQ(\mathcal{M}_\sigma)) \cong HH_*(A_m, \mathcal{M}_\sigma),\] as desired.
\end{proof}

\begin{proposition} \label{prop:SKhiscoinvar} Let $\sigma \in \Braid_{m+1}$ be a braid of index $m+1$ and $\widehat{\sigma} \subset A \times I$ its closure.  Then \[\SKh\left(m(\widehat{\sigma});m-1\right) \cong H_*\left(\mathcal{Q}\left(\mathcal{M}_{\sigma}\right)\right)[n_-]\{(m-1) + n_+-2n_-\}\] as bigraded vector spaces. 
\end{proposition}

\begin{proof}  
As noted in Section \ref{sec:KhS}, the Khovanov-Seidel complex $\mathcal{M}_\sigma$ (and, hence, its coinvariant quotient complex $\cQ(\cM_\sigma)$) has a description in terms of a ``cube of resolutions," by identifying $A_m$ with the trivial TL object and $P_i \otimes {}_i{P}$ with the $i$--th elementary TL object as pictured in Figure \ref{fig:TemperleyLieb}.  The complex $\CKh(m(\widehat{\sigma}) \subset A \times I)$ whose homology is $\SKh(\widehat{\sigma})$ (denoted $\mathcal{C}(m(\widehat{\sigma}) \subset A \times I)$ in \cite[Sec. 2]{GT07060741}--see also \cite{MR2113902}, \cite[Sec. 4]{AnnularLinks}) is also described in terms of a cube of resolutions.  Therefore, to prove that $\SKh(m(\widehat{\sigma});m-1)$ coincides with $\cQ\left(\mathcal{M}_{\sigma}\right)$ (up to the stated grading shift) it suffices to verify that the two cubes of resolutions assign isomorphic bigraded $\F$--vector spaces to the vertices and that,  with respect to this isomorphism, the edge maps agree.

We begin with the vertices of the cube, comparing case-by-case:
\begin{itemize}
	\item the coinvariant quotient, $\cQ\left(\cN_{i_1} \otimes_{A_m} \ldots \otimes_{A_m} \cN_{i_k}\right)$, of the Khovanov-Seidel bimodule associated to a vertex, and
	\item the $\F$--vector space associated to the corresponding vertex in the ``next-to-top" grading of the sutured Khovanov complex, $\CKh(m(\widehat{\sigma});m-1)$.
\end{itemize}  

\begin{figure}
\begin{center}
\resizebox{3in}{!}{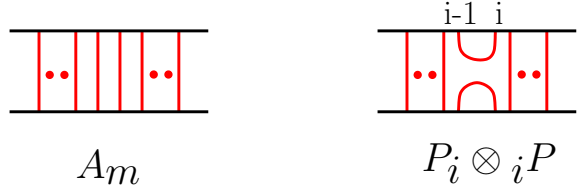}
\end{center}
\caption{The (derived equivalence class of the) dg $A_m$--bimodule associated to an elementary Artin braid $\sigma_i^\pm$ is a (grading-shifted) mapping cone of the bimodules $A_m$ and $P_i \otimes {}_iP$; the former is identified with the trivial TL object (left) and the latter is identified with the $i$th elementary TL object (right).}
\label{fig:TemperleyLieb}
\end{figure}

See Figure \ref{fig:VertexCases} for TL diagrams associated to the various cases.

\begin{figure}
\begin{center}
\resizebox{5in}{!}{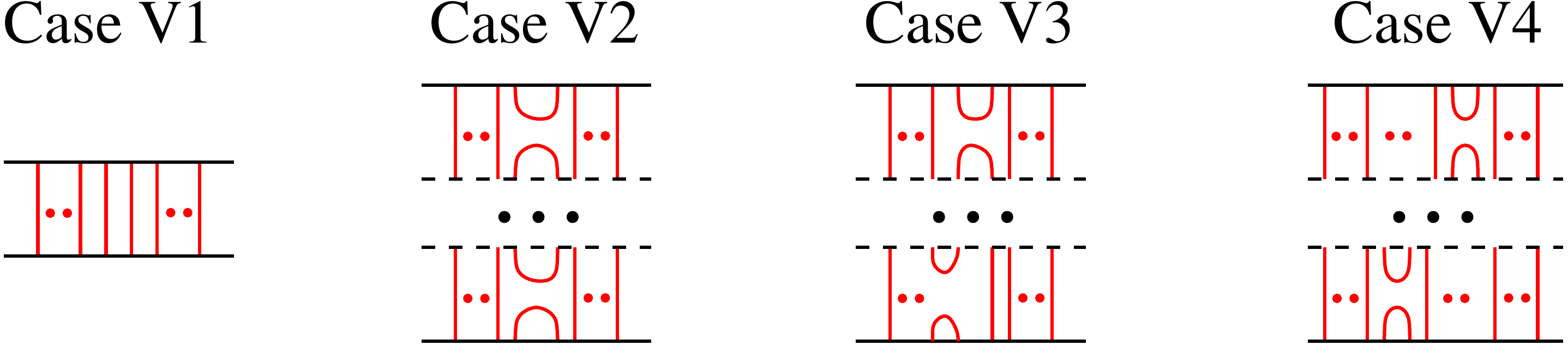}
\end{center}
\caption{We have enumerated above the different types of TL diagrams that can appear at the vertices of a cube of resolutions.}
\label{fig:VertexCases}
\end{figure}

In what follows, we shall always assume that we have used the isomorphism \[M \otimes_{A_m} A_m \cong A_m \otimes_{A_m} M \cong M\] to eliminate extraneous copies of $A_m$ in the tensor product associated to a vertex.  Whenever we refer to a {\em trivial} (resp., {\em nontrivial}) circle in a resolution of the annular closure, we mean a component of the resolution that represents the trivial (resp., nontrivial) element of $H_1(A;\Z/2\Z)$.

{\flushleft {\bf Case V1: $\mathcal{N}_{i_1} \otimes_{A_m} \ldots \otimes_{A_m} \mathcal{N}_{i_k} \cong A_m$}}

{\flushleft \underline{Khovanov-Seidel coinvariant quotient:}}

Note that \[A_m = \bigoplus_{i, j = 0}^m {}_iP_j\] where ${}_iP_j := {}_iP\otimes_{A_m} P_j$.  The coinvariant quotient, $\cQ(A_m)$ is, by definition, the quotient of the (dg) bimodule $A_m$ by the ideal generated by the image of the map $\phi: A_m \otimes A_m \rightarrow A_m$ defined by $\phi(a \otimes b) = ab + ba$. 

Therefore any element of ${}_iP_j$ for $i \neq j$ becomes $0$ when we pass to $\cQ(A_m)$, since each such element can be expressed as $\phi[(i)\otimes (i|\ldots|j)]$.  Furthermore, any element of ${}_iP_i$ of the form $(i|i-1|i) = (i|i+1|i)$ is also $0$ in $\cQ(A_m)$, since we have \[(i|i-1|i) = (i-1|i|i-1) + \phi[(i|i-1)\otimes (i-1|i)].\] 

Now, noting that $(i-1|i|i-1) = (i-1|i-2|i-1) \in A_m$ and iterating, we see that $(i|i-1|i) + (0|1|0)$ is in the image of $\phi$, but $(0|1|0) = 0 \in A_m$.

The only possible nonzero elements of $\cQ(A_m)$ are therefore the idempotents, $\{(i) \in {}_iP_i\}$. But these can only be decomposed as $(i) \otimes (i) \in A_m \otimes A_m$, so we obtain no new relations among them when we pass to the coinvariant quotient. 

We conclude that the vector space associated to a vertex of this type is $(m+1)$--dimensional, with basis given by the idempotents $(0), (1), \ldots, (m)$. In other words, the Khovanov-Seidel coinvariant quotient complex assigns the (grading-shifted) bigraded vector space \[\F_{\left(\vec{v}_h,\vec{v}_q\right)}^{m+1}[n_-]\{(m-1)+(n_+-2n_-)\}\] to the vertex $\vec{v} = (v_1, \ldots, v_k) \in \{-1,1\}^k$ .

{\flushleft \underline{Sutured annular Khovanov complex:}} 

The vertex associated to the closure of the trivial TL object in the cube of resolutions for $\CKh(m(\widehat{\sigma});m-1)$ is also $(m+1)$--dimensional, since the resolution consists of $m+1$ nontrivial circles.  The vector space associated to this vertex in the $f$ grading $m-1$ (the ``next-to-top" one) has a basis given by the $m+1$ enhanced resolutions where exactly one of the circles has been labeled $v_-$, and the rest have been labeled $v_+$.  The bigraded vector space associated to the sutured Khovanov complex at vertex $\vec{v}$ is therefore: \[\F_{\left(\vec{v}_h,(m-1) + \vec{v}_h\right)}^{m+1}[n_-]\{(n_+-2n_-)\}.\] But $\vec{v}_h = \vec{v}_q$ in this case (cf. Lemma \ref{lem:gr}), so the two bigraded vector spaces agree.

{\flushleft \underline{Isomorphism:}}

For $i = 0, 1, \ldots, m$, let $\theta_i$ denote the basis element of $\CKh(m(\widehat{\sigma});m-1)$ described above whose $i$th circle is labeled $v_-$. Then the linear map $\Phi: \cQ(A_m) \rightarrow \mbox{\CKh}(m(\widehat{\sigma});m-1)$:
\[\Phi[(i)] = \left\{\begin{array}{ll}
			\theta_i & \mbox{if $i=0$, and}\\
			\theta_i + \theta_{i-1} & \mbox{otherwise}
			\end{array}\right.\] is an isomorphism of bigraded $\mathbb{F}$--vector spaces (strategically chosen so that the boundary maps along the edges of the cube will agree).

{\flushleft {\bf Case V2: $\mathcal{N}_{i_1} \otimes_{A_m} \ldots \otimes_{A_m} \mathcal{N}_{i_k} = (P_{i_{j_1}} \otimes {}_{i_{j_1}}{P}) \otimes_{A_m} \ldots \otimes_{A_m} (P_{i_{j_n}} \otimes {}_{i_{j_n}}P)$}, where $i_{j_1} = i_{j_n}$}

(We allow here the possibility that $n=1$.)

{\flushleft \underline{Khovanov-Seidel:}}

Rewrite the above tensor product as:
\[P_{i_{j_1}} \otimes ({}_{i_{j_1}}P \otimes_{A_m} P_{i_{j_2}}) \otimes \ldots \otimes ({}_{i_{j_{n-1}}} P \otimes_{A_m} P_{i_{j_n}}) \otimes {}_{i_{j_n}} P\] as in the proof of \cite[Thm. 2.2]{MR1862802}, and note that
\[{}_a P \otimes_{A_m} P_b = \left\{\begin{array}{cl}
	\SpanF\{ (a), (a|a-1|a)\} & \mbox{if $|a-b| = 0$,}\\
	\SpanF\{ (a|b) \} & \mbox{if $|a-b| = 1$, and}\\
	0 & \mbox{if $|a-b|>1$.}\end{array}\right.\]

The corresponding Khovanov-Seidel bimodule is therefore $0$ if there exists $a \in \{1, \ldots, n-1\}$ such that $|i_{j_a} - i_{j_{a+1}}| > 1$. Now suppose there exists no such adjacent pair. Then if the associated TL diagram has $\ell$ closed components, we claim that its closure will have: 
\begin{itemize}
	\item $\ell + 1$ trivial closed circles and
	\item $m-1$ additional nontrivial circles.
\end{itemize}

That there are $\ell + 1$ trivial closed circles in the closure is clear. To see that there are exactly $m-1$ nontrivial circles in the closure, proceed by induction on $n$, the length of the associated TL word. The base case ($n=1$) is quickly verified. If $n > 1$, the assumptions
\begin{enumerate}
	\item There exists no $a \in \{1, \ldots, n-1\}$ such that $|i_{j_a} - i_{j_{a+1}}| > 1$, and
	\item $i_{j_1} = i_{j_n}$
\end{enumerate}
imply that there exists some subword of the TL word that can be replaced by a shorter subword using one of the Jones relations: \eqref{eqn:trivcirc}-\eqref{eqn:kinkleft}, and the claim follows.

Corresponding applications of \cite[Thm 2.2]{MR1862802}\footnote{Note that, since we are using the {\em negative path length} rather than {\em steps-to-the-left} grading, we have $\{-2\}$ rather than $\{1\}$ shifts on the RHS of \cite[Eqns. 2.2-2.4]{MR1862802}.} now allow us to replace the original Khovanov-Seidel bimodule with the quasi-isomorphic bimodule $\bigcirc_{\ell} \,\, (P_{i_{j_1}} \otimes {}_{i_{j_1}}P)\{-(n-1)\}$.

We therefore need only understand $\cQ(P_{i_{j_1}} \otimes {}_{i_{j_1}}P)$. An analysis similar to the one conducted in Case V1 implies that $\cQ(P_{i_{j_1}} \otimes {}_{i_{j_1}}P)$ is free of rank $2$, generated by $(i_{j_1}) \otimes (i_{j_1})$ and $(i_{j_1}) \otimes (i_{j_1}|i_{j_1}-1|i_{j_1})$ (identified with $(i_{j_1}-1|i_{j_1}) \otimes (i_{j_1}|i_{j_1}-1)$ and $(i_{j_1}|i_{j_1}-1|i_{j_1}) \otimes (i_{j_1})$ in the coinvariant quotient module).

We conclude that the (grading-shifted) Khovanov-Seidel coinvariant quotient associated to a vertex of this type is $0$ if there exists $a \in \{1, \ldots, n-1\}$ such that $|i_{j_a} -i_{j_{a+1}}| > 1$, and
\[\bigcirc_{\ell + 1} \mathbb{F}_{(\vec{v}_h,\vec{v}_q-n)}[n_-]\{(m-1) + (n_+ - 2n_-)\}\] otherwise.

{\flushleft \underline{Sutured Khovanov:}}

Suppose that there exists $a \in \{1, \ldots, n-1\}$ such that $|i_{j_a} - i_{j_{a+1}}| > 1$. Then the closure of the corresponding TL diagram can have no more than $m-3$ nontrivial circles.  
The vector space associated to a vertex of this type in this case is therefore $0$, as it is in the Khovanov-Seidel setting.

Now suppose there exists no such adjacent pair. As explained above, if the associated TL diagram has $\ell$ closed components, its closure will have $\ell + 1$ trivial closed circles and $m-1$ additional nontrivial circles.

The $\CKh(m(\widehat{\sigma});m-1)$ vector space therefore has a basis in $1:1$ correspondence with enhanced resolutions whose $(m-1)$ nontrivial circles have all been labeled $v_+$ and the $(\ell + 1)$ trivial circles have been labeled with either $w_{\pm}$. The associated (grading-shifted) bigraded vector space is therefore: \[\bigcirc_{\ell + 1} \mathbb{F}_{(\vec{v}_h,(m-1) + \vec{v}_h)}[n_-]\{n_+ - 2n_-\}.\] Since $n = \sum_{j=1}^k (1-v_j\epsilon_j)$ (cf. Lemma \ref{lem:gr}), the two bigraded vector spaces agree.

{\flushleft \underline{Isomorphism:}}

If there exists $a \in \{1, \ldots, n-1\}$ with $|i_{j_a}-i_{j_{a+1}}| > 1$, the vector spaces (and isomorphism) are trivial.

Otherwise, each of the $(\ell+1)$ trivial circles in the closure of the TL diagram corresponds to either an adjacent pair \[\cdots \otimes ({}_{i_{j_a}}P \otimes_{A_m} P_{i_{j_{a+1}}}) \otimes \cdots\] with $i_{j_a} = i_{j_{a+1}}$ or to the pair of outer terms \[P_{i_{j_1}} \otimes \ldots \otimes {}_{i_{j_n}} P,\] which by assumption also satisfy $i_{j_1} = i_{j_n}$.


In fact, the basis elements of $\cQ(\cM_\sigma)$ are in $1:1$ correspondence with labelings of the corresponding resolved diagram, where each trivial circle is labeled with either the length $0$ path $(i_{j_a})$ or the length $2$ path $(i_{j_a}|i_{j_a} - 1|i_{j_a})$. Similarly, the basis elements of $\CKh(m(\widehat{\sigma});m-1)$ are in $1:1$ correspondence with labelings of the resolved diagram, where each trivial circle is labeled with either a $w_+$ or a $w_-$.

We therefore obtain an isomorphism \[\Phi: \cQ((P_{i_{j_1}} \otimes {}_{i_{j_1}}P) \otimes_{A_m} \ldots \otimes_{A_m} (P_{i_{j_n}} \otimes {}_{i_{j_n}} P)) \rightarrow \CKh(\widehat{\sigma};m-1)\] by identifying the length $0$ (resp., length $2$) path labels on the Khovanov-Seidel side with the $w_+$ (resp., $w_-$) labels on the sutured Khovanov side.

{\flushleft \bf{Case V3: }$\mathcal{N}_{i_1} \otimes_{A_m} \ldots \otimes_{A_m} \mathcal{N}_{i_k} = (P_{i_{j_1}} \otimes {}_{i_{j_1}}{P}) \otimes_{A_m} \ldots \otimes_{A_m} (P_{i_{j_n}} \otimes {}_{i_{j_n}}P)$}, where $|i_{j_1} - i_{j_n}| = 1$.

{\flushleft \underline{Khovanov-Seidel:}}

As in Case V2, the vertex is assigned $0$ if there exists $a \in \{1, \ldots, n-1\}$ such that $|i_{j_a} - i_{j_{a+1}}| > 1$. If there does not exist such an adjacent pair, then the analysis proceeds much as in Case V2. If the TL diagram has $\ell$ closed components, then its closure will have:
\begin{itemize}
	\item $\ell$ trivial circles and 
	\item $(m-1)$ nontrivial circles,
\end{itemize}
by an inductive argument analogous to the one used in Case V2. The corresponding Khovanov-Seidel bimodule is isomorphic to $\langle \ell \rangle (P_{i_{j_1}} \otimes {}_{i_{j_n}} P)\{-(n-1)\}$. 

Since $|i_{j_1} - i_{j_n}| = 1$, $\cQ(P_{i_{j_1}} \otimes {}_{i_{j_n}} P)$ is $1$--dimensional, generated by $(i_{j_1}) \otimes (i_{j_n}|i_{j_1})$.

We conclude that the (grading-shifted) Khovanov-Seidel coinvariant quotient associated to a vertex  $\vec{v}$ of this type is $0$ if there exists $a \in \{1, \ldots, n-1\}$ such that $|i_{j_a} - i_{j_{a+1}}| > 1$, and \[\langle \ell \rangle \F_{(\vec{v}_h,\vec{v}_q-n)}[n_-]\{(m-1)+(n_+-2n_-)\}\] otherwise.

{\flushleft \underline{Sutured Khovanov:}}

As before, the vector space is $0$ if there exists $a \in \{1, \ldots, n-1\}$ such that $|i_{j_a} - i_{j_{a+1}}| > 1$.

Otherwise, suppose that the associated TL diagram contains $\ell$ closed circles. Then, as above, its closure will have $\ell$ trivial circles and $(m-1)$ nontrivial circles. One of these $(m-1)$ nontrivial circles is distinguished by the property that it contains both the {\em cap} of the first elementary TL element and the {\em cup} of the last elementary TL element in the tensor product.

Just as in Case V2, we therefore assign the bigraded vector space: \[\langle \ell\rangle \mathbb{F}_{(\vec{v}_h,(m-1) + \vec{v}_h)}[n_-]\{n_+ - 2n_-\}\] and define a completely analogous isomorphism $\Phi$.

{\flushleft {\bf Case V4: $\mathcal{N}_{i_1} \otimes_{A_m} \ldots \otimes_{A_m} \mathcal{N}_{i_k} = (P_{i_{j_1}} \otimes {}_{i_{j_1}}{P}) \otimes_{A_m} \ldots \otimes_{A_m} (P_{i_{j_n}} \otimes {}_{i_{j_n}}P)$}, where $|i_{j_1} - i_{j_n}| > 1$}

Here we see immediately that the coinvariant quotient of the Khovanov-Seidel bimodule vanishes  because there are no nonzero paths in $A_m$ between vertices separated by distance more than $1$, and the sutured annular Khovanov complex in $f$ grading $m-1$ vanishes since the closure of the corresponding TL object can have no more than $(m-3)$ nontrivial circles.
\vskip 10pt
We now verify that the edge maps agree. Again, this is a check of a finite number of cases, each related to one of those pictured in Figure \ref{fig:Cases} by a finite sequence of cyclic permutations and horizontal reflections.  By applying Lemma \ref{lem:QCyclic}, we may assume without loss of generality that the merge or split defining the edge map occurs in the final (top) term. 

\begin{figure}
\begin{center}
\resizebox{5.5in}{!}{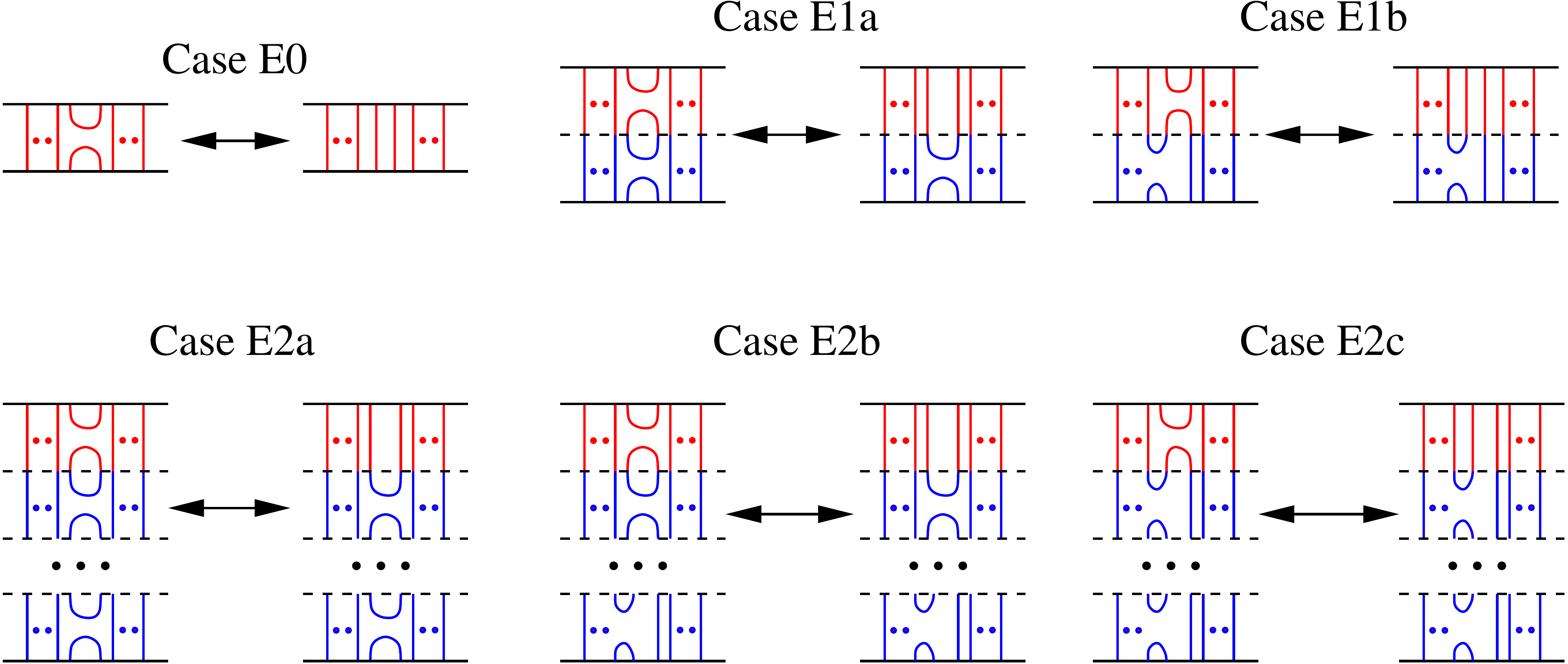}
\end{center}
\caption{We have enumerated (modulo vertical cyclic permutation and horizontal reflection) all possible merges/splits occurring in {\em nontrivial} edge maps in the cube-of-resolution models for the Khovanov-Seidel and sutured Khovanov complexes. The relevant merge/split is pictured in red, and the cases are enumerated according to which TL object(s) are (cyclically) adjacent to the merge/split. Note that the edge map is necessarily $0$ if the local configuration does not fall into one of the cases above, since one or both of the vertices it connects must be $0$ by the argument given in Cases V2--V4.}
\label{fig:Cases}
\end{figure}

On the sutured Khovanov side, each of the edge maps corresponds to a map induced by a merge/split cobordism of one of the following types:
\begin{itemize}
	\item {\em Type I}: Two nontrivial circles $\longleftrightarrow$ One trivial circle. This corresponds to merge/split maps of the form $V \otimes V \longleftrightarrow W$ in the language of \cite[Sec. 2]{GT07060741}. Example E0 is of this type.
	\item {\em Type II}: Two trivial circles $\longleftrightarrow$ One trivial circle. This corresponds to merge/split maps of the form $W \otimes W \longleftrightarrow W$. Examples E1a and E2a are of this type.
	\item {\em Type III}: One trivial and one nontrivial circle $\longleftrightarrow$ One nontrivial circle. This corresponds to merge/split maps of the form $V \otimes W \longleftrightarrow V$. Examples E1b, E2b, and E2c are of this type.
\end{itemize}

While describing Cases V2-V4, we noted that every trivial circle in the annular closure of a TL element arises as a result of a (cyclically) adjacent pair \[\cdots \otimes {}_{i} P_{i} \otimes \cdots := \cdots \otimes ({}_{i} P \otimes_{A_m} P_{i}) \otimes \cdots\] in the associated Khovanov-Seidel tensor product. Moreover, ``labelings" of the trivial circles (by either the length $0$ or length $2$ path) on the Khovanov-Seidel side correspond, via the isomorphism $\Phi$, to labelings of the trivial circles (by either $w_+$ or $w_-$). It can now be seen directly that the Khovanov-Seidel maps $\beta_i$ and $\gamma_i$ behave exactly like the sutured Khovanov merge/split maps in all cases. We include an explicit verification of this in Cases E0, E1a, and E1b. The other cases are similar.

{\flushleft {\bf Case E0: $P_i \otimes {}_iP\longleftrightarrow A_m $}}

{\flushleft ``$\rightarrow$":} 

On the Khovanov-Seidel side, recall from Case V2 that $\cQ(P_i \otimes {}_iP)$ has basis given by \[\{(i) \otimes (i), (i)\otimes (i|i-1|i)\}\] and from Case V1 that $\cQ(A_m)$ has basis given by the idempotents \[\{(0), \ldots, (m)\}.\] Moreover, 
\begin{eqnarray*}
	\cQ(\beta_i)[(i) \otimes (i)] &=& (i)\\
	\cQ(\beta_i)[(i) \otimes (i|i-1|i)] &=& 0
\end{eqnarray*} 

On the sutured Khovanov side, the split map sends the generator whose trivial circle is labeled $w_+$ (see Case V2) to $\theta_i + \theta_{i-1}$ (see Case V1). Under the isomorphism $\Phi$ described in Cases V1 and V2, these maps agree.

{\flushleft ``$\leftarrow$":}

The only basis elements on the Khovanov-Seidel side with nontrivial image under $\cQ(\gamma_i)$ are the idempotents $(i-1)$ and $(i+1)$, both sent to \[(i-1|i) \otimes (i|i-1) = (i+1|i) \otimes (i|i+1) = (i) \otimes (i|i-1|i) \in \cQ(P_i \otimes {}_iP).\] On the sutured Khovanov side, the merge map sends the generators $\theta_{i-1}$ and $\theta_i$ to the generator whose single trivial circle has been labeled $w_-$  and whose nontrivial circles have all been labeled $v_+$. Again, these maps agree under the isomorphism $\Phi$.

{\flushleft {\bf Case E1a: $((P_{i} \otimes {}_iP) \otimes_{A_m} (P_i \otimes {}_iP)) \longleftrightarrow (P_i \otimes {}_i P)$}}

{\flushleft ``$\rightarrow$":} 

On the Khovanov-Seidel side, let ${}_iP_i$ denote $P_i \otimes_{A_m} {}_iP$ and recall from Case V2 that $\cQ(P_i \otimes {}_iP_i \otimes {}_iP)$ has basis given by $\{(i) \otimes a \otimes b\,\,|\,\, a, b \in \{(i), (i|i-1|i)\}\}$, and the map $\cQ(\beta_i)$ is given by multiplication of the last two factors:
\begin{eqnarray*}
\cQ(\beta_i)[(i) \otimes (i) \otimes (i)] &=& (i) \otimes (i)\\
\cQ(\beta_i)[(i) \otimes (i|i-1|i) \otimes (i)] = \cQ(\beta_i)[(i) \otimes (i) \otimes (i|i-1|i)] &=& (i) \otimes (i|i-1|i)\\
\cQ(\beta_i)[(i) \otimes (i|i-1|i) \otimes (i|i-1|i)] &=& 0.
\end{eqnarray*}
On the sutured Khovanov side, the map is given by multiplying the labels $w_{\pm}$ on the two trivial circles, which merge to form one. Now note that the isomorphism $\Phi$ identifies a generator on the Khovanov-Seidel side whose second or third tensor factor is labeled $(i)$ (resp., $(i|i-1|i)$) with a generator on the sutured Khovanov side whose first or second trivial circle is labeled $w_+$ (resp., $w_-$). Moreover, the multiplication (merge) map in both settings is the multiplication in $\F[x]/x^2$ via the identification of $(i) \leftrightarrow w_+$ with  $1$ and $(i|i-1|i) \leftrightarrow w_-$ with $x$ in $\F[x]/x^2$. The Khovanov-Seidel and sutured Khovanov maps therefore agree.

{\flushleft ``$\leftarrow$":} On the Khovanov-Seidel side, the map $\cQ(\gamma_i)$ is given by:
\begin{eqnarray*}
\cQ(\gamma_i)[(i) \otimes (i)] &=& (i) \otimes (i) \otimes (i|i-1|i) + (i) \otimes (i|i-1|i) \otimes (i)\\
\cQ(\gamma_i)[(i) \otimes (i|i-1|i)] &=& (i) \otimes (i|i-1|i) \otimes (i|i-1|i)
\end{eqnarray*}
which agrees with the split map on the sutured Khovanov side under the correspondence $\Phi$, as described above.

{\flushleft {\bf Case E1b: $((P_{i-1} \otimes {}_{i-1}P) \otimes_{A_m} (P_i \otimes {}_iP)) \longleftrightarrow (P_{i-1} \otimes {}_{i-1}P)$}}

{\flushleft ``$\rightarrow$":} On the Khovanov-Seidel side, again let ${}_{i-1}P_i$ denote ${}_{i-1}P \otimes_{A_m} P_i$ and recall from Case V3 that $\cQ(P_{i-1} \otimes {}_{i-1}P_i \otimes {}_{i}P)$ is generated by  $(i-1) \otimes (i-1|i) \otimes (i|i-1)$, and \[\cQ(\beta_{i})[(i-1) \otimes (i-1|i) \otimes (i|i-1)] = (i-1) \otimes (i-1|i-2|i-1).\] On the sutured Khovanov side, we have a single nontrivial circle splitting into one trivial and one nontrivial circle, and the split map sends the generator $v_+$ to the generator $v_+ \otimes w_-$. Under the isomorphism $\Phi$ from Cases V1 and V2, the two maps therefore agree.

{\flushleft ``$\leftarrow$":} On the Khovanov-Seidel side, we have 
\begin{eqnarray*}
	\cQ(\gamma_i)[(i-1) \otimes (i-1)] &=& (i-1) \otimes (i-1|i) \otimes (i|i-1)\\
	\cQ(\gamma_i)[(i-1) \otimes (i-1|i-2|i-1)] &=& 0.
\end{eqnarray*} This agrees, under the isomorphism $\Phi$, with the sutured Khovanov merge map, which sends the generator $w_+ \otimes v_+$ to the generator $v_+$ and the generator $w_- \otimes v_+$ to $0$.

\end{proof}
\begin{proposition} \label{prop:HHiscoinvar}
Let $\sigma= \sigma_{i_1}^\pm \cdots \sigma_{i_k}^\pm$ be a braid. Then \[HH_*(A_m, \cM_{\sigma}) \cong H_*(\cQ(\cM_{\sigma})).\]
\end{proposition}
\begin{proof}
Given any resolution
\[\xymatrix{\mathcal{R}(A_m) \ar[r]^-{\partial} & A_m} := \xymatrix{(\cdots \ar[r]^-{\partial_2} R_2 & R_1 \ar[r]^-{\partial_1} & R_0) \ar[r]^-{\partial} & A_m}\] of $A_m$ by projective right $A_m^e := A_m \otimes A_m^{op}$ modules, $HH(A_m, \cM_\sigma)$ is, by Definition \ref{defn:HochHomClassical}, the homology of the complex $\mathcal{R}(A_m) \otimes_{A_m^e} \mathcal{M}_\sigma.$

Note furthermore that $\mathcal{R}(A_m) \otimes_{A_m^e} \cM_\sigma$ has the structure of a double complex, whose ``horizontal" differentials are of the form $\partial_* \otimes_{A_m^e} \Id$, where $\partial_*$ is the differential in $\mathcal{R}(A_m)$ and whose ``vertical" differentials are of the form $\Id \otimes_{A_m^e} d_*$, where $d_*$ is the internal differential in the complex $\mathcal{M}_{\sigma}$. For example, the double complex for $\mathcal{M}_{\sigma_i^+}$ looks like:
\[\xymatrix{\cdots R_2 \otimes_{A_m^e} (P_i \otimes {}_iP)\{1\}\ar[r]^-{\partial_2 \otimes \Id} \ar[d]^-{\Id \otimes \beta_i}  & R_1 \otimes_{A_m^e} (P_i \otimes {}_iP)\{1\} \ar[r]^-{\partial_1 \otimes \Id} \ar[d]^-{\Id \otimes \beta_i} & R_0 \otimes_{A_m^e} (P_i \otimes {}_iP)\{1\}\ar[d]^-{\Id \otimes \beta_i} \\
\cdots R_2 \otimes_{A_m^e} A_m\{1\} \ar[r]^-{\partial_2 \otimes \Id} & R_1 \otimes_{A_m^e} A_m\{1\} \ar[r]^-{\partial_1 \otimes \Id} & R_0 \otimes_{A_m^e} A_m\{1\}}\]

Let $d_h:= \partial_* \otimes_{A_m^e} \Id$ (resp., $d_v := \Id \otimes_{A_m^e} d_*$) denote the horizontal (resp., vertical) differential on the complex $\mathcal{C}_\sigma := \mathcal{R}(A_m) \otimes_{A_m^e} \mathcal{M}_\sigma$. There is a corresponding spectral sequence converging to \[HH(A_m, \mathcal{M}_\sigma) := H_*(\mathcal{C}_\sigma, d_h + d_v)\] whose $E^2$ term is \[H_*(H_*(\mathcal{C}_\sigma, d_h), d_v).\]

Moreover, by choosing the bar resolution:
\[\cR(A_m) \rightarrow A_m := \xymatrix{(\cdots A_m^{\otimes 4} \ar[r] & A_m^{\otimes 3} \ar[r] & A_m^{\otimes 2}) \ar[r] & A_m}\]
we see that $\cQ(\cM_{\sigma})$ is precisely the chain complex whose underlying vector space is $H_0(\mathcal{C}_\sigma, d_h)$ and whose differential is the induced differential, $d_v$, on the quotient. Hence: \[H_*(H_0(\mathcal{C}_\sigma, d_h), d_v) \cong H_*(\cQ(\cM_{\sigma}))\]

Now we claim that $H_n(\mathcal{C}_\sigma, d_h) = 0$ for all $n \neq 0$. Since the induced differential on the $E^2$ page must have a nonzero horizontal degree shift, the claim implies the desired result: \[HH(A_m, \cM_\sigma) \cong H_*(\cQ(\cM_\sigma)),\] since the spectral sequence must therefore collapse at the $E^2$ stage.

To prove the claim, observe first that the chain complex $(\mathcal{C}_\sigma, d_h)$ splits as a direct sum of chain complexes, one for each vertex of the cube of resolutions for $\cM_\sigma$. Accordingly, each subcomplex is of the form:
\[\cR(A_m)  \otimes_{A_m^e} (\cN_{i_1} \otimes_{A_m} \ldots \otimes_{A_m} \cN_{i_k}),\] where \[\cN_{i_j} := \left\{\begin{array}{c} A_m\\P_{i_j} \otimes {}_{i_j}P\end{array}\right.\] depending on the vertex of the cube; hence, the homology of each such subcomplex is \[HH(A_m, \cN_{i_1} \otimes_{A_m} \ldots \otimes_{A_m} \cN_{i_k}).\]

Now note that $P_{k} \otimes {}_{\ell}P$ is a projective $A_m^e$ module for any $k, \ell \in \{0, \ldots, m\}$, since \[A_m \otimes A_m = \bigoplus_{k,\ell = 0}^m P_k \otimes {}_\ell P.\] Since the tensor product functor is exact on projective modules, the Hochschild homology of a projective bimodule is concentrated in degree $0$. We conclude that whenever at least one of the tensor factors $\cN_{i_j}$ is of the form $P_{i_j} \otimes {}_{i_j} P$, we have \[H_n(\cR(A_m) \otimes_{A_m^e} (\cN_{i_1} \otimes_{A_m} \ldots \otimes_{A_m} \cN_{i_k})) = 0\] for $n \neq 0$, as desired.

The only remaining computation is $HH(A_m, \cN_{i_1} \otimes_{A_m} \ldots \otimes_{A_m} \cN_{i_k})$ in the case $\cN_{i_1} = \ldots = \cN_{i_k} = A_m$, i.e., $HH(A_m, A_m)$.

Rather than computing $HH(A_m, A_m)$ directly, we will exploit a relationship between $A_m$ and another dg algebra, which we will call $B_m$, whose Hochschild homology is easy to compute for algebraic reasons (its homology is {\em directed}). The motivation here comes from the geometric interpretation of the algebras $A_m$ and $B_m$ in terms of the Fukaya category of a certain Lefschetz fibration, described in \cite{GT10014323}, \cite[Chp. 20]{SeidelBook}, and recalled briefly in \cite[Sec. 3.5]{QuiverAlgebras}.

The definition of the dg algebra $B_m$ can (and will) be given combinatorially, but symplectic geometers would do well to keep the following description in mind. Let \[S=\{(x,y,z)\in\mathbb{C}^3\,|\,x^2+y^2=p(z)\},\] where $p$ is a polynomial of degree $m+1$ whose roots are exactly the points
${\bf 0}, \ldots, {\bf m}$ pictured in Figure \ref{fig:Qbasis}. Then projection to the $z$
coordinate gives a Lefschetz fibration $\pi:S\to \mathbb{C}$, for which the arcs
$p_1,\dots,p_m$ pictured in Figure \ref{fig:Qbasis} are matching paths and lift to Lagrangian
spheres $P_1^S,\dots,P_m^S\subset S$. Meanwhile, the vanishing paths
$p_0,q_0,q_1,\dots,q_m$ determine Lefschetz thimbles $P_0^S,Q_0^S,\dots,Q_m^S$.
These are respectively spherical and exceptional objects of the Fukaya
category $\mathcal{F}(\pi)$ \cite{SeidelBook}, in which
$\mathrm{End}(P_0^S\oplus \dots\oplus P_m^S)\simeq A_m$, whereas
$\mathrm{End}(Q_0^S\oplus\dots\oplus Q_m^S)\simeq B_m^{Kh}$.
The construction carried out here is equivalent to expressing the elements
of the exceptional collection $Q_0^S,\dots,Q_m^S$ as twisted complexes
built out of the generators $P_0^S,\dots,P_m^S$. See also \cite[Sec. 3]{QuiverAlgebras}. Note that in \cite{QuiverAlgebras}, the subscript $m$ was dropped from the notation for $B_m$, $B_m^{Kh}$, and $B_m^{HF}$.

\begin{figure}
\begin{center}
\resizebox{3in}{!}{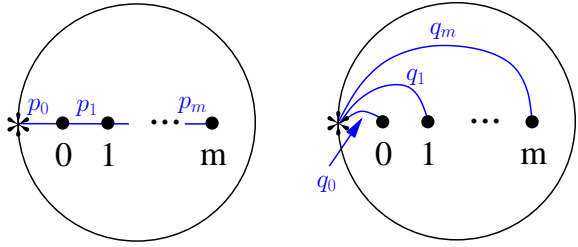}
\end{center}
\caption{The matching and vanishing paths $p_i$ and $q_i$ in the base of the Lefschetz fibration $\pi:S\to\C$.}
\label{fig:Qbasis}
\end{figure}

The dg algebra $B_m$ can be described combinatorially by: \[B_m := \bigoplus_{i,j=0}^m \mbox{Hom}_{A_m}(Q_i, Q_j) = \bigoplus_{i,j=0}^m {}_iQ_j,\] where \[Q_i := \xymatrix{P_0 \ar[r]^-{\cdot (0|1)} & P_1 \ar[r]^-{\cdot (1|2)} & \ldots \ar[r]^-{\cdot (i-1|i)} & P_i},\] and
 
\[{}_iQ_j := \vcenter{\xymatrix{{}_0P_0 \ar[r]^-{\cdot (0|1)} & {}_0P_1 \ar[r]^-{\cdot (1|2)} & \ldots \ar[r]^-{\cdot (j-1|j)} & {}_0P_j \\
				{}_1P_0 \ar[r] \ar[u]^-{(0|1) \cdot} & {}_1P_1 \ar[r] \ar[u] & \ldots \ar[r] \ar[u] & {}_1P_j \ar[u]\\
				\cdots \ar[u]^-{(1|2) \cdot} \ar[r] & \cdots \ar[u] \ar[r] & \cdots \ar[u] \ar[r] & \cdots \ar[u] \\
				{}_iP_0 \ar[r]  \ar[u]^-{(i-1|i) \cdot} & {}_iP_1 \ar[r] \ar[u] & \ldots \ar[r] \ar[u] & {}_iP_j\ar[u] }}.\]
\vskip 10pt				
In the above, ${}_iP_j:= \mbox{Hom}_{A_m}(P_i,P_j)$ denotes the subspace of $A_m$ generated by paths beginning at $i$ and ending at $j$, and the horizontal (resp., vertical) maps are given by right (resp., left) multiplication by the appropriate length-one path, as indicated in the first row (resp., column). 

Let $B_m^{Kh} := H_*(B_m)$. It is shown in \cite[Lem. 3.12]{QuiverAlgebras} that $B_m^{Kh}$ is isomorphic to the subalgebra of {\em lower triangular} $(m+1) \times (m+1)$ matrices over $H^*(S^1;\F) \cong \F[x]/x^2$ with diagonal entries in $\F$ (\cite[Rmk. 3.13]{QuiverAlgebras}).

We now use the category equivalences: \[\xymatrix@C=6pc@R=3pc{D_\infty(A_m-A_m^{op}) \ar@/^1pc/[r]^-{\mathcal{F}^e} & D_\infty(B_m-B_m^{op})\ar@/^1pc/[l]^-{\mathcal{G}^e} \ar@/^1pc/[r]^-{\mbox{Induct}_\phi} & D_\infty(B_m^{Kh}-(B_m^{Kh})^{op}) \ar@/^1pc/[l]^-{\mbox{Rest}_\phi}}\] provided by the proof of \cite[Prop. 3.15]{QuiverAlgebras} and the $A_\infty$ quasi-isomorphism $\phi: B_m \rightarrow B_m^{Kh}$ guaranteed by \cite[Lem. 3.12]{QuiverAlgebras}.

Explicitly, the equivalence \[\cF^e: D_\infty(A_m-A_m^{op}) \rightarrow D_\infty(B_m-B_m^{op})\] is given by \[\cF^e(M) := Q^*\widetilde{\otimes}_{A_m} M \widetilde{\otimes}_{A_m} Q,\footnote{The ordinary tensor product agrees with the $A_\infty$ tensor product here, since $Q$ (resp., $Q^*$) is a complex of projective left (resp., right) modules over $A_m$.}\] where $Q := \bigoplus_{i = 0}^m Q_i$,  and $Q^*:= \bigoplus_{i=0}^m \mbox{Hom}_{A_m}(Q_i, A_m) = \bigoplus_{i=0}^m {}_iQ,$ with \[{}_iQ := \xymatrix{{}_0P & \ar[l]_-{(0|1)\cdot} {}_1P & \ar[l]_-{(1|2)\cdot} \cdots & \ar[l]_-{(i-1|i)\cdot} {}_iP},\] and the equivalence 
\[\mbox{Induct}_\phi: D_\infty(B_m - B_m^{op}) \rightarrow D_\infty(B_m^{Kh}-(B_m^{Kh})^{op})\] is given by \[\mbox{Induct}_\phi(M) := B_m^{Kh} \widetilde{\otimes}_{B_m} \,\,M\,\, \widetilde{\otimes}_{B_m} B_m^{Kh},\] where the right/left $B_m$--module structure on $B_m^{Kh}$ above is obtained in the natural way by using the quasi-isomorphism $\phi: B_m \rightarrow B_m^{Kh}$. 

The proof of \cite[Prop. 3.15]{QuiverAlgebras} tells us that \[Q \widetilde{\otimes}_{B_m} Q^* = Q \otimes_{B_m} Q^* = A_m \in D_\infty(A_m - A_m^{op}),\] so an application of Lemma \ref{lem:HHTensor} and the observation that $\cF^e(A_m) = B_m$ implies that $HH(A_m, A_m) \cong HH(B_m,B_m)$.

Similarly, the formality of $B_m$ implies that \[B_m^{Kh} \widetilde{\otimes}_{B_m^{Kh}} B_m^{Kh} = B_m^{Kh} \otimes_{B_m^{Kh}} B_m^{Kh} = B_m \in D_\infty(B_m-B_m^{op})\]
(where the first tensor factor is viewed as a $B-B^{Kh}$ bimodule and the second as a $B_m^{Kh}-B_m$ bimodule), so another application of Lemma \ref{lem:HHTensor} and the observation that $\mbox{Induct}_\phi(B_m) = B_m^{Kh}$ implies that $HH(B_m, B_m) \cong HH(B_m^{Kh},B_m^{Kh}).$

It remains to compute $HH(B_m^{Kh},B_m^{Kh})$. Since $B_m^{Kh}$ is an associative algebra, we use the classical definition of Hochschild homology given in Definition \ref{defn:HochHomClassical}.  
Accordingly, we choose the ``composable" (projective) bar resolution of $B_m^{Kh}$ as a left $(B_m^{Kh})^e := B_m^{Kh} \otimes (B_m^{Kh})^{op}$ module: 
\[\xymatrix{\cdots \ar[r] & \bigoplus_{i_1, \ldots, i_4} \left({}_{i_1}Q^{Kh}_{i_2} \otimes {}_{i_2}Q^{Kh}_{i_3} \otimes {}_{i_3}Q^{Kh}_{i_4}\right) \ar[r] & \bigoplus_{i_1, i_2, i_3} \left({}_{i_1}Q^{Kh}_{i_2} \otimes {}_{i_2}Q^{Kh}_{i_3}\right)},\] which, after tensoring over $(B_m^{Kh})^e$ with $B_m^{Kh}$, yields the ``composable" cyclic bar complex:
\[\xymatrix{\cdots \ar[r]^-{\partial_3} & \bigoplus_{i_1, i_2, i_3} \left({}_{i_1}Q^{Kh}_{i_2} \otimes {}_{i_2}Q^{Kh}_{i_3} \otimes {}_{i_3}Q^{Kh}_{i_1}\right) \ar[r]^-{\partial_2} & \bigoplus_{i_1, i_2} \left({}_{i_1}Q^{Kh}_{i_2} \otimes {}_{i_2}Q^{Kh}_{i_1}\right) \ar[r]^-{\partial_1} & \bigoplus_{i} {}_iQ_i}\] where, as usual, \[\partial_k(b_1 \otimes \ldots \otimes b_{k+1}) := (b_1b_2 \otimes \ldots \otimes b_{k+1}) + (b_1 \otimes b_2b_3 \otimes \ldots \otimes b_{k+1}) +  \ldots + (b_2 \otimes \ldots \otimes b_{k+1}b_1).\]

But, as noted previously, $B_m^{Kh}$ is {\em directed}:
\[{}_iQ^{Kh}_j := H_*(\mbox{Hom}_{A_m}(Q_i,Q_j)) = 0 \mbox{ whenever $i < j$},\] so the composable cyclic bar complex reduces to:

\[\xymatrix{\cdots \ar[r]^-{\partial_3} & \bigoplus_{i} \left({}_iQ^{Kh}_i\right)^{\otimes 3} \ar[r]^-{\partial_2} & \bigoplus_{i} \left({}_iQ^{Kh}_i\right)^{\otimes 2} \ar[r]^-{\partial_1} & \bigoplus_{i} {}_iQ^{Kh}_i.}\]

Moreover, for each $i \in \{0, \ldots, m\}$ and $k \in \Z^+$, $({}_iQ^{Kh}_i)^{\otimes k}$ is $1$--dimensional, spanned by the tensor product, $\left({}_i\Id_i\right)^{\otimes k}$, of $k$ copies of the idempotent ${}_i\Id_i \in {}_iQ_i^{Kh}$. Therefore, the restricted map \[\partial_k: \left({}_iQ_i^{Kh}\right)^{\otimes k+1} \rightarrow \left({}_iQ_i^{Kh}\right)^{\otimes k}\] is given by \[\partial_k({}_i\Id_i \otimes \ldots \otimes {}_i\Id_i)= \left\{\begin{array}{cl}
					0 & \mbox{if $k$ is odd, and}\\
					\left({}_i\Id_i\right)^{\otimes k} & \mbox{if $k$ is even.}\end{array}\right.\]

It follows that \[HH(B_m^{Kh},B_m^{Kh}) \cong HH_0(B_m^{Kh}, B_m^{Kh}) \cong \bigoplus_{i=0}^m {}_iQ_i^{Kh}\] has dimension $m+1$. But we calculated in Case V1 of the proof of Proposition \ref{prop:SKhiscoinvar} that $HH_0(A_m, A_m) = \cQ(A_m)$ also has dimension $m+1$. Since \[m+1 = \mbox{dim}(HH_0(A_m,A_m)) \leq \mbox{dim}(HH(A_m,A_m)) = \mbox{dim}(HH(B_m^{Kh},B_m^{Kh}))= m+1,\] it follows that $HH_n(A_m,A_m) = 0$ unless $n=0$, as desired.
\end{proof}

\section{Ozsv{\'a}th-Szab{\'o} spectral sequence}
Theorem \ref{thm:SKhisHH}, when combined with  \cite[Thm. 6.1]{QuiverAlgebras} and Proposition \ref{prop:HHbordsut}, a  generalization of \cite[Thm. 14]{BorderedBimodules}, implies Theorem \ref{thm:OzSzspecseq}, a (portion of the) result originally obtained in \cite{GT07060741} and \cite{AnnularLinks} by adapting Ozsv{\'a}th-Szab{\'o}'s argument from \cite{MR2141852} to Juh{\'a}sz's sutured Floer homology setting.

As before, $\sigma \in \Braid_{m+1}$ denotes a braid, and $\widehat{\sigma} \subset A \times I$ denotes its annular closure (with $A \times I$ identified in the standard way with the complement of the braid axis, $K_B$), and $m(\widehat{\sigma})$ its mirror. In the following, 
\begin{enumerate}
\item $SFH(\boldSigma(\widehat{\sigma}))$ denotes the sutured Floer homology (\cite{MR2253454}) of the double cover, $\boldSigma(\widehat{\sigma})$, of the product sutured manifold $A \times I$, branched along $\widehat{\sigma} \subset A \times I$, as described in \cite{AnnularLinks},
\item $\widehat{HFK}(\boldSigma(\widehat{\sigma}),\widetilde{K}_B)$ denotes the knot Floer homology (\cite{MR2065507},\cite{GT0306378}) of the preimage, $\widetilde{K}_B \subset \boldSigma(\widehat{\sigma})$, of $K_B$, and
\item $V = \F_{\left(\frac{1}{2},0\right)} \oplus \F_{\left(-\frac{1}{2},-1\right)}$ is a ``standard" $2$--dimensional bigraded vector space  (the subscripts on the summands indicate their (Alexander, Maslov) bigrading).
\end{enumerate}

\begin{theorem} \label{thm:OzSzspecseq} Let $\sigma \in \Braid_{m+1}$. There exists a filtered chain complex whose associated graded homology is isomorphic to $\SKh(m(\widehat{\sigma});m-1)$ and whose total homology is isomorphic to the ``next-to-bottom" Alexander grading of 
\[SFH(\boldSigma(\widehat{\sigma})) \cong \left\{\begin{array}{cl}
	\widehat{HFK}(\boldSigma(\widehat{\sigma}), \widetilde{K}_B) & \mbox{when $m+1$ is even, and}\\
	\widehat{HFK}(\boldSigma(\widehat{\sigma}), \widetilde{K}_B) \otimes V & \mbox{when $m+1$ is odd,}\end{array}\right.\]
where $\widetilde{K}_B$ denotes the preimage of the braid axis, $K_B$, in $\boldSigma(\widehat{\sigma})$, the double cover of $S^3$ branched along $\widehat{\sigma}$.
\end{theorem}

In particular, there exists a spectral sequence from the next-to-top filtration grading of $\SKh(m(\widehat{\sigma}))$ to the next-to-bottom Alexander grading of $\mbox{SFH}(\boldSigma(\widehat{\sigma}))$.

\begin{remark} We have chosen to match the Alexander grading normalization conventions of \cite{MR2253454} rather than \cite{CombHFK} in the statement of Theorem \ref{thm:OzSzspecseq}, since these fit more naturally with Zarev's conventions in \cite{GT09081106}; we made the opposite choice in  \cite[Sec. 7]{QuiverAlgebras}.  Note also that in \cite[Sec. 7]{QuiverAlgebras} we asserted an isomorphism between $HH(B_m^{HF},\cM_\sigma^{HF})$ and the {\em next-to-top} (rather than {\em next-to-bottom}, as asserted here) Alexander grading of $\widehat{HFK}(\boldSigma(\widehat{\sigma}), \widetilde{K}_B)$. The two assertions are equivalent, because of the symmetry of knot Floer homology described in \cite[Sec. 3.5]{MR2065507}.
\end{remark}

We begin by stating the following generalization of \cite[Thm. 14]{BorderedBimodules} (replacing the pointed matched circles of \cite{GT08100687} with the more general non-degenerate arc diagrams of \cite[Defn. 2.1]{GT09081106}), which we need for the proof of Theorem \ref{thm:OzSzspecseq}. Its proof follows quickly from Theorems 1 and 2 of Zarev's \cite{GT09081106}:

\begin{proposition} \label{prop:HHbordsut} Let $(Y,\Gamma,\cZ, \phi)$ be a bordered sutured manifold in the sense of \cite[Defn. 3.5]{GT09081106}, with $\cZ = -\cZ_1 \amalg \cZ_2$. If there is a diffeomorphism $\Psi: G(-\cZ_1) \rightarrow G(\cZ_2)$ for which $\Psi|_{{\bf Z}_1}: {\bf Z}_1 \rightarrow {\bf Z}_2$ is orientation-preserving, then there is an identification of the Hochschild homology of the $A_\infty$ bimodule \[{}_{\cA(\cZ_2)}\widehat{BSDA}_M(Y,\Gamma)_{\cA(\cZ_2)}\] (whose construction is described in \cite[Sec. 8.4]{GT09081106}) with the sutured Floer homology of  $(Y', \Gamma')$, the sutured manifold obtained by self-gluing $(Y,\Gamma,\cZ,\phi)$ along $F(-\cZ_1) \sim_\Psi F(\cZ_2)$ as described in \cite[Sec. 3.3]{GT09081106}:

\[SFH(Y',\Gamma') \cong HH(\cA(\cZ_2), {}_{\cA(\cZ_2)}\widehat{BSDA}_M(Y,\Gamma)_{\cA(\cZ_2)}).\]

Moreover, the isomorphism appropriately identifies the ``moving-strands" grading of $\cA(\cZ_2)$ with the ``Alexander" grading of $SFH(Y',\Gamma')$ relative to $F(\cZ_2)$.
\end{proposition}

See the end of the proof for an explicit identification of these two gradings.

\begin{proof}[Proof of Proposition \ref{prop:HHbordsut}]
Let $m$ denote the number of arcs in $\cZ_2$. Then $\cA(\cZ_2)$ splits as a direct sum:
\[\cA(\cZ_2) := \bigoplus_{k=0}^m\cA(\cZ_2,k),\] over $k$--moving strands algebras as in \cite[Sec. 2.2]{GT09081106}, and there is a corresponding decomposition of the bimodule:
\[{}_{\cA(\cZ_2)}\widehat{BSDA}_M(Y,\Gamma)_{\cA(\cZ_2)} := \bigoplus_{k=0}^m {}_{\cA(\cZ_2,k)}\widehat{BSDA}_M(Y,\Gamma,k)_{\cA(\cZ_2,k)}.\] 

To compactify notation, we shall denote $\cA(\cZ_2,k)$ by $\cA_k$ and ${}_{\cA(\cZ_2,k)}\widehat{BSDA}_M(Y,\Gamma,k)_{\cA(\cZ_2,k)}$ by $\widehat{BSDA}_k(Y)$ for the remainder of the proof.

As  alluded to in \cite[Sec. 2.4.3]{BorderedBimodules} we have, for each $k$, an equivalence of categories \[\cG: D_\infty(\cA_k-\cA_k^{op}) \rightarrow D_\infty(\cA_k \otimes \cA_k^{op}),\] and after replacing $\cG(\widehat{BSDA}_k(Y))$ with a quasi-isomorphic dg model as in \cite[Prop. 2.4.1]{BorderedBimodules}, we obtain that the complex \[CC(\cA_k, \widehat{BSDA}_k(Y)):=\cG(\widehat{BSDA}_k(Y)) \widetilde{\otimes}_{\cA_k^{op} \otimes \cA_k} \cA_k.\] is chain isomorphic to the Hochschild complex described in Definition \ref{defn:HochHomAinfty}.\footnote{Here we use the canonical correspondence between left $A$--modules and right $A^{op}$--modules.}
But Zarev tells us, in \cite[Defn. 8.3]{GT09081106}, that $\cG(\widehat{BSDA}_k(Y))$ is quasi-isomorphic to the $(k,k)$--strand module $\widehat{BSA}(Y,\Gamma,-\cZ_1 \amalg \cZ_2,k,k)$ (cf. \cite[Sec. 1.2]{GT09081106}), and \cite[Thm. 2]{GT09081106} tells us that the bimodule associated to 
\[\mbox{id}_{\cZ_2} := (F(\cZ_2) \times I, \Lambda \times I),\] the identity bordered sutured cobordism  from $F(\cZ_2)$ to itself (cf. \cite[Defn. 1.3]{GT09081106}), is quasi-isomorphic to the underlying algebra: \[{}_{\cA(\cZ_2)}\widehat{BSDA}_M(\mbox{id}_{\cZ_2})_{\cA(\cZ_2)} \cong  \cA(\cZ_2) \in D_\infty(\cA(\cZ_2)-(\cA(\cZ_2))^{op}).\] The desired result then follows from Zarev's pairing theorem, \cite[Thm. 1]{GT09081106}. Explicitly, for each $k$, the total homology of

\begin{eqnarray*}
CC(\cA_k, \widehat{BSDA}_k(Y)) &=&  \cG(\widehat{BSDA}_k(Y)) \widetilde{\otimes}_{\cA_k^{op} \otimes \cA_k} \cA_k \\
&=& \widehat{BSA}(Y,\Gamma,\cZ,k,k) \boxtimes \widehat{BSD}(\mbox{id}_{\cZ_2},k,k) \\
\end{eqnarray*}
agrees with $SFH(Y',\Gamma';k)$, where $SFH(Y',\Gamma';k)$ denotes the homology of the subcomplex of $CFH(Y',\Gamma')$ supported in those (relative) Spin$^c$ structures $\mathfrak{s} \in$ Spin$^c(Y',\Gamma')$ satisfying \[\langle c_1(\mathfrak{s}), [F(\cZ_2)]\rangle = \chi(F(\cZ_2)) - |{\bf Z}_2| + 2k.\] Here, $c_1(\mathfrak{s})$ denotes the first Chern class of $\mathfrak{s}$ (\cite[Sec. 3]{MR2390347}) with respect to a canonical trivialization (cf. \cite[Prf. of Thm. 1.5]{MR2390347}) of $v_0^{\perp}$, and $|{\bf Z}_2|$ denotes the number of oriented line segments (``platforms") of $\cZ_2$.

Informally, $SFH(Y',\Gamma';k)$ is the ``$k$-from-bottom" Alexander grading of $SFH(Y',\Gamma')$ with respect to $[F(\cZ_2)]$, the homology class of $F(\cZ_2)$.
\end{proof}

\begin{proof}[Proof of Theorem \ref{thm:OzSzspecseq}] 

In \cite[Thm. 6.1]{QuiverAlgebras}, we 
\begin{itemize}
	\item construct a filtration on $\cM_\sigma^{HF}$ (defined in \cite[Sec. 3]{QuiverAlgebras}), the $1$--moving--strand $A_\infty$ bordered Floer bimodule over the $1$--moving--strand algebra $B_m^{HF}$ (described in \cite[Rmk. 4.1]{QuiverAlgebras}), and
	\item identify the associated graded complex, $\mbox{gr}(\cM_\sigma^{HF})$ (an $A_\infty$ bimodule over $B_m^{Kh} = \mbox{gr}(B_m^{HF})$ using \cite[Thm. 5.1]{QuiverAlgebras}) with $\cM_\sigma^{Kh}$ (defined in \cite[Sec. 2]{QuiverAlgebras} and recalled briefly in the proof of Proposition \ref{prop:HHiscoinvar}).
\end{itemize}

As in \cite[Lem. 2.12]{QuiverAlgebras}, the $A_\infty$ Hochschild complex, \[CC(B_m^{HF}, \cM_\sigma^{HF}) := \left(\bigoplus_{k=0}^\infty \left(B_m^{HF}\right)^{\otimes k} \otimes \cM_\sigma^{HF}, \partial\right),\] inherits a filtration, and as in \cite[Lem. 2.15]{QuiverAlgebras} we see that the associated graded of the Hochschild complex is the Hochschild complex of the associated graded: \[\mbox{gr}(CC(B_m^{HF},\cM_\sigma^{HF})) = CC(B_m^{Kh},\cM_\sigma^{Kh}) \in D_\infty(B_m^{Kh}-(B_m^{Kh})^{op}).\]

\begin{figure}
\begin{center}
\resizebox{1.5in}{!}{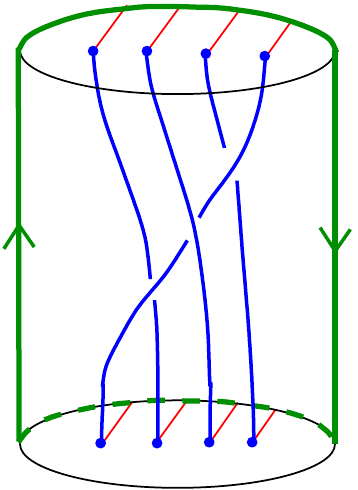}
\end{center}
\caption{Lift the sutures (green) and parameterizing arcs (red) to obtain a bordered-sutured manifold structure on $\boldSigma(\sigma)$, the double cover of $D_m \times I$ branched along $\sigma$ (blue).}
\label{fig:BordSutBraid}
\end{figure}

Now consider the bordered sutured manifold, $(\boldSigma(\sigma), \Gamma, \cZ, \phi)$, obtained from the sutured manifold $D_m \times I$ by lifting the arcs and sutures pictured in Figure \ref{fig:BordSutBraid} to a bordered sutured manifold structure on the double-cover of $D_m \times I$ branched along $\sigma$.  The first author proves in \cite{GT10014323} (cf. \cite[Sec. 4]{QuiverAlgebras}) that $B_m^{HF}$ is isomorphic to $\mathcal{A}(\mathcal{Z}_2,1)$ and $\cM_\sigma^{HF}$ is quasi-isomorphic to ${}_{\cA(\cZ_2,1)}\widehat{BSDA}_M(\boldSigma(\sigma),\Gamma,1)_{\cA(\cZ_2,1)}$. Letting $\boldSigma(\widehat{\sigma})$ denote the double-branched cover of the product sutured manifold $(A \times I, \Gamma')$ branched along the closure, $\widehat{\sigma}$, of $\sigma$, and $\widetilde{\Gamma'}$ the preimage of $\Gamma'$, Proposition \ref{prop:HHbordsut} then implies:
\[SFH(\boldSigma(\widehat{\sigma}), \widetilde{\Gamma'};1) \cong HH(B_m^{HF}, \cM_\sigma^{HF}),\]
so $HH(B_m^{HF}, \cM_\sigma^{HF})$ agrees with the next-to-bottom ($\cong$ next-to-top) Alexander grading of 
\[SFH(\boldSigma(\widehat{\sigma})) \cong \left\{\begin{array}{cl}
	\widehat{HFK}(\boldSigma(\widehat{\sigma}), \widetilde{K}_B) & \mbox{when $m+1$ is even, and}\\
	\widehat{HFK}(\boldSigma(\widehat{\sigma}), \widetilde{K}_B) \otimes V & \mbox{when $m+1$ is odd.}\end{array}\right.\]
See \cite[Rmk. 7.1]{QuiverAlgebras} for an explanation of the extra tensor factor of $V$ in the odd index case.

To obtain the desired result, we now need only argue that the homology of the Hochschild complex \[CC(B_m^{Kh}, \cM_\sigma^{Kh}) = \mbox{gr}(CC(B_m^{HF},\cM_\sigma^{HF}))\] agrees with $\SKh(m(\widehat{\sigma});m-1)$.

But, noting that $\cM_\sigma^{Kh}$ is the image of $\cM_\sigma$ under the derived equivalences 
\[D_\infty(A_m-A_m^{op}) \longleftrightarrow D_\infty(B_m-B_m^{op}) \longleftrightarrow D_\infty(B_m^{Kh}-(B_m^{Kh})^{op})\] provided by \cite[Prop. 3.15]{QuiverAlgebras}, another two applications of Lemma \ref{lem:HHTensor} as in the proof of Proposition \ref{prop:HHiscoinvar} tells us that 
\[HH(B_m^{Kh},\cM_\sigma^{Kh}) \cong HH(A_m,\cM_\sigma),\] 
and Theorem \ref{thm:SKhisHH} then tells us that the homology of $CC(B_m^{Kh},\cM_\sigma^{Kh})$ is $\SKh(m(\widehat{\sigma});m-1)$, as desired.

\end{proof}

\bibliography{HochHom}
\end{document}